\documentclass[11pt]{amsart}
\usepackage{amsmath}
\usepackage{amsthm}
\usepackage{amssymb}
\usepackage{amscd}
\usepackage{amsfonts}
\usepackage{amsbsy}
\usepackage{epsfig}

\newtheorem{theorem}{Theorem}
\newtheorem{remark}{Remark}
\newtheorem{proposition}{Proposition}
\newtheorem{lemma}[proposition]{Lemma}
\newtheorem{corollary}[proposition]{Corollary}

\newtheorem*{claim*}{Claim}
\newtheorem{definition}[proposition]{Definition}

\def\ie{{\em i.e.,\ }}
\def\eg{{\em e.g.\ }}

\newfont\bbf{msbm10 at 12pt}
\def\eps{\varepsilon}

\def\R{{\mathbb R}}
\def\C{{\mathbb C}}
\def\N{{\mathbb N}}
\def\Z{{\mathbb Z}}
\def\K{{\mathbb K}}

\def\L{{\mathbb L}}

\def\diam{\mbox{\rm diam} }
\def\area{\mbox{\rm area} }
\def\orb{\mbox{\rm orb}}
\def\Crit{\mbox{\rm Crit}}
\def\crit{\text{\tiny crit}}

\def\le{\leqslant}
\def\ge{\geqslant}

\def\1{ {\hbox{{\it 1}} \!\! I} }

\def\feig{{\mbox{\tiny feig}}}
\def\qfeig{{\mbox{\tiny q-feig}}}
\def\cfeig{{\mbox{\tiny feig}}}
\def\MP{{\mbox{\tiny MP}}}
\def\TM{{\mbox{\tiny TM}}}
\def\Var{{\mbox{Var}}}

\newdimen\AAdi%
\newbox\AAbo%
\def\AArm{\fam0 }
\def\AAk#1#2{\setbox\AAbo=\hbox{#2}\AAdi=\wd\AAbo\kern#1\AAdi{}}%
\def\AAr#1#2#3{\setbox\AAbo=\hbox{#2}\AAdi=\ht\AAbo\raise#1\AAdi\hbox{#3}}%

\def\BBone{{\AArm 1\AAk{-.8}{I}I}}%

\newcommand {\CA}{{\mathcal A}}

\newcommand {\CC}{{\mathcal C}}

\newcommand {\CH}{{\mathcal D}}

\newcommand {\CJ}{{\mathcal J}}

\newcommand {\CL}{{\mathcal L}}

\newcommand {\CP}{{\mathcal P}}

\newcommand {\CR}{{\mathcal R}}

\newcommand {\CW}{{\mathcal W}}

\def\llb{[\![} \def\rrb{]\!]}

\newcommand{\disp}{\displaystyle}
\newcommand{\8}{\infty}

\def\m1{{-1}}

\newcommand{\ninf}{{n\rightarrow\8}}

\def\S{\Sigma}
\def\s{\sigma}

\newcommand{\wt}{\widetilde}

\newcommand{\ul}[1]{\underline{#1}}

\begin{document}
\synctex=1

\title[Renormalization and thermodynamics]
{Renormalization, thermodynamic formalism  and quasi-crystals in subshifts.}
\author{Henk Bruin and Renaud Leplaideur}
\date{Version of \today}
\thanks{Part of this research was supported by a Scheme 3 (ref 2905) 
visitor grant of 
the London Mathematical Society and a visiting professorship at the
University of Brest. 
}

\begin{abstract}
We examine thermodynamic formalism for a class of renormalizable dynamical
systems which in the symbolic space is generated by the Thue-Morse 
substitution, and in complex dynamics by the Feigenbaum-Coullet-Tresser map.
The basic question answered is whether fixed points $V$ of a renormalization
operator $\CR$ acting on the space of potentials are such that
the pressure function
$\gamma \mapsto \CP(-\gamma V)$ exhibits phase transitions.
This extends the work by Baraviera, Leplaideur and Lopes on the 
Manneville-Pomeau map, where such phase transitions were indeed detected.
In this paper, however, the attractor of renormalization is a Cantor set
(rather than a single fixed point), which admits various classes of 
fixed points of $\CR$, some of which do and some of which do not exhibit
phase transitions.
In particular, we show it is possible to reach, as a ground state, 
a quasi-crystal before temperature zero by freezing a dynamical system. 
\end{abstract}

\maketitle

\section{Introduction}\label{sec:intro}
\subsection{Background}
Phase transitions are a central theme in statistical mechanics and 
probability theory. 
In the physics/probability approach the dynamics is not very relevant and just 
emerges as a by-product of the invariance by translation. The main difficulty is the geometry of the $\Z^{d}$ lattice. Considering an interacting particle systems such as the Ising model (see \eg \cite{gallavotti1, georgii}), it is possible to find a measure (called Gibbs measure) that maximizes the probability of
obtaining a configuration with minimal free energy associated to a Hamiltonian. This is done considering a finite box and fixing the conditions on its boundary. Then letting the size of the box tend to infinity, the sequence of Gibbs
measures have a set of accumulation points. If this set varies non-continuously with respect to the parameters (including the temperature), 
then the system is said to exhibit a \emph{phase transition}. 

In contrast, the time evolution of the system is the central theme in 
dynamics systems. The theory of thermodynamic formalism has been imported 
into hyperbolic dynamics in the 70's, essentially by Sinai, Ruelle and Bowen. 
Gradually, authors started to extend this theory to the non-uniformly 
hyperbolic case, sometimes
applying \emph{inducing techniques} that are also important in this paper. 
Initially, phase transitions have been less central in dynamical systems, 
but the development of the theory of ergodic optimization since the 2000's has 
naturally led mathematicians to introduce (or rather rediscover) the 
notion of ground states. The question of phase transitions arises naturally
in this context. 

Note that vocabulary used in statistical mechanics is sometimes quite different
from that used in dynamical system. What in statistical mechanics vocabulary 
is called a ``freezing'' transition, such as occur in 
Fisher-Felderhof models (see \eg \cite{fisher}), corresponds 
in the mathematical vocabulary to the Manneville-Pomeau map or the
shift with Hofbauer potential (see \eg \cite{wang} or 
\cite[Exercise 5.8 on page 98]{ruelle} and also \cite{gallavotti2}).
 
Renormalization is an over-arching theme in physics and dynamics,
including thermodynamic formalism, see \cite{CGU} for modern results
in the direction.
The system that we study in this paper is related to 
 cascade of doubling period phenomenon and 
the infinitely renormalizable maps {\it \`a la} Feigenbaum-Coullet-Tresser,
which is on the boundary of chaos (see \eg \cite{schuster}).
Instead of the freezing transitions, the system has its
equilibrium state (at phase transition) supported on a Cantor set rather 
than in a fixed point or a periodic orbit. Stated in physics terminology, we prove that it is possible to reach a quasi-crystal as a ground state before 
temperature zero by freezing a dynamical system (see Theorems \ref{theo-thermo-a<1} and \ref{theo-super-MP}). This issue is related to a question due to van Enter (see \cite{vEM}). The original question was for $\Z^{2}$-actions, but we hope that ideas here may be exported to this more complicated case. 

Returning to the mathematical motivation, the present paper takes the work of 
\cite{baraviera-leplaideur-lopes} a step further.
We investigate the connections between phase transition in the full $2$-shift, renormalization for potentials, renormalization for maps (in complex dynamics) and substitutions in the full $2$-shift. Here the attractor of renormalization
is a Cantor set, rather than a single point, and its thermodynamic
properties turn out to be strikingly different.

We recall that Bowen's work \cite{Bowen} on thermodynamic formalism showed 
that every subshift of finite type with H\"older continuous potential $\phi$
admits a unique equilibrium state (which is a Gibbs measure). Moreover, the
pressure function $\gamma \mapsto \CP(-\gamma \phi)$ is real analytic and  there are no phase transitions.
This is also known as the Griffiths-Ruelle theorem. 
Hofbauer \cite{Hnonuni} was the first (in the dynamical systems world) 
to find continuous non-H\"older potentials
for the full two-shift $(\S, \s)$
allowing a phase transition at some $t = t_{0}$.

A geometric interpretation of Hofbauer's example leads naturally
to the Manneville-Pomeau map $f_\MP:[0,1] \to [0,1]$ defined as
$$
f_\MP(x) = \left\{ \begin{array}{ll}
\frac{x}{1-x} & \text{ if } x \in [0, \frac12], \\
2x-1 & \text{ if } x \in (\frac12, 1],
\end{array}\right.
$$
with a neutral fixed point at $0$.
This map admits a local renormalization $\psi(x) = \frac{x}{2}$ which 
satisfies

\begin{equation}
\label{equ-renorm-MP}
f_\MP^2 \circ \psi(x) = \psi \circ f_\MP(x) \qquad \text{ for all } 
x \in  [0, \frac12].
\end{equation}
If we differentiate Equation~\eqref{equ-renorm-MP}, take logarithms 
and subtract $\log \psi'\equiv \log\frac12$ from both sides 
of the equality, we find 

\begin{equation}
\label{equ2-renom-mp-pot}
\log|f_\MP'| =  \log |f_\MP'| \circ f_\MP \circ \psi(x) + \log |f_\MP'| \circ \psi(x).
\end{equation}

Passing to the shift-space again (via the itinerary map for the
 standard partition $\{[0,\frac12], \ (\frac12,1]\}$),
 we are  naturally led to renormalization  in the shift. 
Of prime importance are the solutions of the equation 
\begin{equation}
\label{equ-renorm-shift}
\s^{2}\circ H=H\circ \s,
\end{equation}
which replaces the renormalization scaling $\psi$ in \eqref{equ-renorm-MP}. 
Equation~\eqref{equ2-renom-mp-pot} leads to an operator $\CR$ defined by 
$$
\CR(V) = V \circ \s \circ H + V \circ H.
$$
In \cite{baraviera-leplaideur-lopes}, the authors investigated the case of the substitution
$$
 H_\MP:\left\{ \begin{array}{l} 0 \to 00, \\ 1 \to 01, \end{array} \right. 
$$
which has a unique fixed point $ 0^\8$, corresponding to the neutral fixed
point $0$ of $f_\MP$. In \cite{baraviera-leplaideur-lopes}, the map $H_\MP$ was not presented as a substitution but  we emphasize here  (and it is an improvement because it allows more general studies) 
that it indeed is; more generally, any constant-length $k$ substitution
solves Equation~\eqref{equ-renorm-shift} (with $\s^{k}$ instead of $\s^{2}$). 
It is also shown in \cite{baraviera-leplaideur-lopes} that the operator $\CR$ 
fixes the Hofbauer potential 
$$
V(x):=\log\frac{n+1}{n} \quad  \text{ if }
x \in [0^{n}1] \setminus [0^{n+1}1], \quad n>0.
$$ 
Moreover, the lift of $\log f'_\MP$ belongs to the \emph{stable set} of the Hofbauer potential. This fact is somewhat mysterious because the substitution $H_\MP$ \emph{is not} the lift of the scaling function $\psi:x\mapsto x/2$. 

\bigskip
In this paper we focus on the Thue-Morse substitution; see \eqref{eq:HTM} 
for the definition. 
It is one of the simplest substitutions satisfying  the renormalization equality  \eqref{equ-renorm-shift} and contrary to $H_\MP$, the attractor for the 
Thue-Morse substitution, 
say $\K$, is not a periodic orbit but a Cantor set.
Yet similarly to the Manneville-Pomeau fixed point, 
$\s:\K \to \K$ has zero entropy and is uniquely ergodic. This is one way to define quasi-crystal in ergodic theory. 

The thermodynamic formalism 
 for the Thue-Morse substitution is much more complicated, and more 
interesting, than for the Manneville-Pomeau substitution. This is because
Cantor structure of the attractor admits a more intricate recursion 
behavior of nearby points (although it has zero entropy) 
characterized by what we call ``accidents'' in 
Section~\ref{subsec-theo-vtilde}, which
are responsible for the lack of phase transitions for the ``good'' 
fixed point for $\CR$, 
This allows much more chaotic shadowing than  when the 
attractor of the substitution is a periodic orbit. 
We want to emphasize here that our results are extendible to more 
general substitutions, but to get the main ideas across, we focus
on the Thue-Morse shift in this paper.

\subsection{Statements of results}
The Thue-Morse substitution
\begin{equation}\label{eq:HTM} 
H := H_\TM: \left\{ \begin{array}{l} 0 \to 01 \\ 1 \to 10 \end{array} \right.
\end{equation}
has two fixed points
$$
\rho_1 = 1001\ 0110\ 1001\ 0110\ 01\dots \quad \text{ and } \quad 
\rho_0 = 0110\ 1001\ 0110\ 1001\ 10\dots
$$
Let $\K = \overline{\cup_n \s^n(\rho_0)} = \overline{\cup_n \s^n(\rho_1)}$ be the corresponding subshift of the full shift $(\S, \s)$ 
on two symbols. 
The renormalization equation~\eqref{equ-renorm-shift}  holds in $\S$: $H \circ \s = \s^2 \circ H$, and we define the {\em renormalization operator} acting on functions $V:\S \to \R$
as
$$
(\CR V)(x) = V \circ \s \circ H(x) + V \circ H(x).
$$
We consider the usual metric on $\S$: $d(x,y) = \frac1{2^n}$ if $n = \min\{ i \ge 1 : x_i \neq y_i \}$. 
This distance is sometimes graphically represented as follows: 
\begin{figure}[ht]
\unitlength=6mm
\begin{picture}(12,4.5)(-2,0)
\put(-1,1.5){\small$x_{0}=y_{0}$}
\put(0,2){\line(1,0){5}}
\put(5,2){\line(1,1){1}} \put(5,2){\line(1,-1){1}}
\put(5,1){\dashbox{0.2}(0.1,3)}\put(4.5,4.3){\small $n-1$} 
\put(4,0){\small $x_{n-1}=y_{n-1}$}
\put(6,3){\line(1,0){3}}
\put(9.5,3){\small $y$}
\put(6,1){\line(1,0){3}}
  \put(9.5,1){\small $x$} 
\end{picture}
\caption{The sequence $x$ and $y$ coincide for digits $0$ up to $n-1$ and then split.}\label{fig:distance}
\end{figure}

Note that $d(H^nx, H^ny) = d(x,y)^{2^n}$: if $x$ and $y$ coincide for $m$ digits, then $H^{n}(x)$ and $H^{n}(y)$ coincide for $2^nm$ digits. 
\\[2mm]
The first two results deal with the continuous fixed points for the renormalization operator $\CR$. The main issue is to determine fixed points and their \emph{weak stable leaf}, namely the potentials attracted by the considered fixed point by iterations of $\CR$. 

The second series of results deals with the thermodynamical formalism; we study if some class of potentials related to weak  
stable leaf of the fixed points, exhibit a phase transition. In particular, Theorem~\ref{theo-super-MP} is related to a question of Van Enter et al.\ (see \eg \cite{vEM,vEMZ}) asking whether it is possible to reach a quasi-crystal by freezing a system before 
zero temperature. 

The last result (Theorem~\ref{theo-thermo-Vu}) returns to the geometrical dynamics and shows the main difference between the Thue-More case and the Manneville-Pomeau case. Due to the Cantor structure of the attractor of the substitution, there exist non-continuous but locally constant  (on $\K$ up to  a finite number of points) fixed points for $\CR$. As the Hofbauer potential represents the logarithm  of the derivative of an affine approximation of the Manneville-Pomeau map, one of these potentials, $V_{u}$, represents the logarithm of the derivative of an affine approximation to the Feigenbaum-Coullet-Tresser map $f_{feig}:\C \to \C$. The main difference with the Manneville-Pomeau case is that here, $V_{u}$ has no phase transition whereas $-\log |f'_{feig}|$ has. 

\subsubsection{ Results on continuous fixed points for $\CR$}
Define the one-parameter family of potentials
\begin{equation}\label{eq:Uc}
U_c = \left\{ \begin{array}{rl}
c & \text{ on } [01], \\
-c & \text{ on } [10], \\
0 & \text{ on } [00] \cup [11].
\end{array}\right.
\end{equation}
It is easy to verify that $U_{c}$ is a fixed point of $\CR$.
Given a fixed function $V:\S\to\R$, the \emph{variation on $k$-cylinders} $\Var_{k}(V)$ is defined as
$$
\Var_{k}(V):=\max\{|V(x)-V(y)|,\ x_{j}=y_{j}\,\mbox{for }j=0,\ldots,k-1\}.
$$
The condition $\sum_{k=1}^{\8}\Var_{k}(W)<\8$ holds if \eg $W$ is H\"older continuous. 

\begin{theorem}\label{theo-super-uc}
If $W$ is a continuous fixed point of $\CR$ on $\K$ such that 
$$\sum_{k=1}^{\8}\Var_{k}(W)<\8,$$
then $W = U_c$ for $c = W(\rho_0)$.
\end{theorem}

As for the Hofbauer case, we produce a non-negative continuous fixed point for $\CR$ with a well-defined \emph{weak stable set}\footnote{In \cite{baraviera-leplaideur-lopes} it was proven that $\CR^{n}(V)$ converges to the fixed point $\wt V$; here we only get convergence in the Cesaro sense.}

\begin{theorem}\label{theo-vtilde}
There exists a unique function $\wt V$, such that 
$\wt V = \lim_{m \to \infty} \frac1m\sum_{k=0}^{n-1}\CR^kV$ for every continuous $V$ satisfying $V(x) = \frac1n+o(\frac1n)$ if $d(x,\K) = 2^{-n}$. 
Moreover $\wt V$ is $\CR$-invariant, continuous  and 
positive except on $\K$: $\frac1{2n} \le \wt V(x) \le \frac{1}{n-1}$
if $d(x,\K) = 2^{-n}$.
\end{theorem}

\subsubsection{Results on Thermodynamic Formalism}
We refer to Bowen's book \cite{bowen} for the background on thermodynamic formalism, equilibrium states and Gibbs measures in $\S$. 
However, in contrast to Bowen's book, our potentials are not H\"older-continuous. 

For a given potential $W:\S\to\R$, the pressure of $W$ is defined by 
$$\CP(W):=\sup\{h_{\mu}(\s)+\int W\,d\mu\},$$ 
where $h_{\mu}(\s)$ is the Kolmogorov entropy of the invariant probability measure $\mu$.  The supremum is a maximum in $\S$ whenever $W$ is continuous. Any measure realizing this maximum is called an equilibrium state. 
We want to study the regularity of the function $\gamma\mapsto
\CP(-\gamma W)$. For simplicity, this function will also be
denoted by $\CP(\gamma)$. If $\CP(\gamma)$ fails to be analytic, we speak of a 
phase transition. We are in particular interested in the special
phase transition as $\gamma \to \8$: easy and classical computations
show that $\CP(\gamma)$ has an asymptote of the form $-a\gamma+b$ as
$\gamma \to \8$. By an \emph{ultimate phase transition} we mean that
$\CP(\gamma)$ reaches its asymptote at some $\gamma'$. In this case,
there cannot be another phase transition for larger $\gamma$, hence
{\em ultimate}.
Then, by a convexity argument, $\CP(\gamma)=-a\gamma+b$  for any $\gamma\ge \gamma'$. 
One of the main motivations for studying
 ultimate phase transitions is that the quantity $a$ satisfies
$$
a=\inf\left\{\int W\,d\mu,\ \mu \text{ is a shift-invariant probability measure} \right\}.
$$
An example of an ultimate phase transition for rational maps can be 
found in \cite{makarov-smirnov}. 
The Manneville-Pomeau map is another classical example. 

\begin{theorem}[No phase transition]
\label{theo-thermo-a>1}
Let $a > 1$ and $V:\S\rightarrow \R$ be a continuous function satisfying $V(x) = \frac1{n^{a}}+o(\frac1{n^{a}})$ if $d(x,\K) = 2^{-n}$.
Then, for every $\gamma\ge 0$, there exists a unique equilibrium state associated to $-\gamma V$ and it gives
positive mass to every open set. The pressure function $\gamma\mapsto\CP(\gamma)$ is analytic and positive on $[0, \infty)$, although it converges to zero as $\gamma \to \8$. 
\end{theorem}

\begin{theorem}[Phase transition]
\label{theo-thermo-a<1}
Let $a \in (0,1)$ and $V:\S\rightarrow \R$ be a continuous function satisfying $V(x) = \frac1{n^{a}}+o(\frac1{n^{a}})$ if $d(x,\K) = 2^{-n}$. Then there exists $\gamma_{1}$ such that for every $\gamma>\gamma_{1}$ the unique equilibrium state for $-\gamma V$ is the unique invariant measure $\mu_{\K}$ supported on $\K$. For $\gamma<\gamma_{1}$, there exists a unique equilibrium state associated to $-\gamma V$ and it gives positive mass to every open set in $\S$. 
The pressure function $\gamma\mapsto\CP(\gamma)$  is positive and analytic 
on $[0,\gamma_{1})$.
\end{theorem}

These results show that case $a=1$ (\ie the Hofbauer potential)
is the border between the regimes with and without phase transition.
Whether there is a phase transition for the case $a=1$ (\ie the fixed point $\tilde V$) or in other words the analog of the Hofbauer potential, discussed in \cite{baraviera-leplaideur-lopes}, is much more subtle. 
We intend to come back to this question in a later paper.

The full shift $(\S,\s)$ can be interpreted geometrically 
by a degree $2$ covering of the circle. The Manneville-Pomeau map can be 
viewed this way; it is expanding except for a single (one-sided) 
indifferent fixed point.
When dealing with the Thue-Morse shift, it is natural to look
for a circle covering with an indifferent Cantor set.

\begin{theorem}
\label{theo-super-MP}
There exist $\CC^1$ maps $f_{a}:[0,1]\circlearrowleft$, 
 semi-conjugate to the full $2$-shift and  expanding everywhere except on a Cantor set $\wt\K$, such that $\wt\K$ is conjugate to $\K$ in $\S$ and 
 if $a \in (0,1)$, then $-\gamma\log f'_{a}$ has an ultimate phase transition. 
\end{theorem}

Another geometric realization of the Thue-Morse shift 
and the prototype of renormalizability in one-dimensional dynamics
is the Feigenbaum map.
This quadratic interval map $f_\qfeig$ has zero entropy, but when complexified
it has entropy $\log 2$. 
Moreover, it is conjugate to another
analytic degree $2$ covering map on $\C$,
which we call $f_\cfeig$, that is fixed by the Feigenbaum renormalization 
operator 
$$
\CR_\feig f = \Psi^{-1} \circ f^2 \circ \Psi
$$ 
where $\Psi$ is linear $f$-dependent holomorphic contraction.
Arguments from complex dynamics
give that $\CP(-\gamma \log |f'_\cfeig|) = 0$ for all $\gamma \ge 2$,
 see Proposition~\ref{prop:feig_PT}.
Because $h_{top}(f_\cfeig) = \log 2$ on its Julia set,
the potential $-\gamma_1 \log |f'_\cfeig|$ has a phase transition 
for some $\gamma_1 \in (0,2]$.
When lifted to symbolic space, $-\log|f'_\cfeig|$
produces an unbounded potential $V_\cfeig$ which is
fixed by $\CR$.
We can find a potential $V_u$, which is constant on 
$$
(\s\circ H)^{k}(\S)\setminus (\s\circ H)^{k+1}(\S)
$$
for each $k$
such that $\| V_\cfeig - V_u\|_\infty < \infty$ and 
analyze the thermodynamic properties of $V_u$.
Although $\CP(-\gamma_1 V_\cfeig) = 0$ for some $\gamma_1 \le 2$,  
it is surprising to see that the potential $V_u$ exhibits no phase transition.
We emphasize here an important difference with the Manneville-Pomeau case,
where both the potential $-\gamma \log |f'_{MP}|$ and its countably piecewise
version, the Hofbauer potential, which is constant on cylinder sets 
$(H_{MP})^{k}(\S)\setminus(H_{MP})^{k+1}(\S) = [0^{2k+1}1]$,
undergo a phase transition.

\begin{theorem}[No phase transition for unbounded fixed point $V_u$]
\label{theo-thermo-Vu}
The unbounded potential $V_u$ given by
$$
V_u(x) = \alpha(k-1) \quad
\text{ for } \quad 
x \in (\s \circ H)^k(\S) \setminus (\s \circ H)^{k+1}(\S)
$$
is a fixed point of $\CR$ for any $\alpha \in \R$. If $\alpha < 0$, then
for every $\gamma\ge 0$, there exists a unique equilibrium state for $-\gamma V_{u}$. It gives positive mass to any open set in $\S$.
The pressure function $\gamma \mapsto \CP(-\gamma V_u)$
is analytic  and positive for all $\gamma \in [0, \infty)$.
\end{theorem}

The exact definition of the equilibrium state for this unbounded potential 
can be found in Subsection~\ref{subsec-unbounded2}. 

\subsection{Outline of the paper}
In Section~\ref{sec:renorm} we prove Theorems~\ref{theo-super-uc} and \ref{theo-vtilde}. In the first subsection we recall some  and prove other results on the Thue-Morse substitution and its associated attractor $\K$.

In Section~\ref{sec:thermo} we study the thermodynamic formalism and prove 
Theorems~\ref{theo-thermo-a<1}, \ref{theo-super-MP} and \ref{theo-thermo-Vu}. 
This section uses extensively the theory of local thermodynamic 
formalism defined in \cite{leplaideur1} 
and developed in further works of the author. 
Finally, in the Appendix, we explain the relation between the
Thue-Morse shift and the Feigenbaum map, and state and prove Proposition~\ref{prop:feig_PT}.

\section{Renormalization in the Thue-Morse shift-space}\label{sec:renorm}

\subsection{General results on the Thue-Morse shift-space}\label{subsec-genresult-TM}

Let $\s:\S \to \S$ be the full shift on $\S = \{ 0, 1 \}^\N$.
If $x = x_0x_1x_2x_4\dots \in \S$, let $[x_0\dots x_{n-1}]$ denote the $n$-cylinder containing $x$, and let $\bar x_i = 1-x_i$ be our notation for the opposite symbol.

Recall that $\rho_0$ and $\rho_1$ are the fixed points of the Thue--More substitution, and that 
$\K = \overline{\orb_\s(\rho_0)} = \overline{\orb_\s(\rho_1)}$ is a uniquely ergodic and zero-entropy subshift. We denote by $\mu_{\K}$ its invariant measure.

We give here some properties for the Thue-Morse sequence that can be found in \cite{BLRS, B, dekking, dLV}. 
\begin{enumerate}
\item
 {\em Left-special} words
(\ie words $w$ such that both $0w$ and $1w$ appear in $\K$)
are prefixes of $H^k(010)$ or of $H^k(101)$
for some $k \ge 0$. 
\item
 {\em Right-special} words
(\ie words $w$ such that both $w0$ and $w0$ appear in $\K$)
are suffixes of $H^k(010)$ or of $H^k(101)$
for some $k \ge 0$. 
\item 
{\em Bispecial} words (\ie words $w$ such that both left and right-special)
are precisely the words $\tau_k := H^k(0)$, $\bar \tau_k := H^k(1)$, 
$\tau_k\bar \tau_k \tau_k = H^k(010)$ and $\bar \tau_k \tau_k \bar \tau_k = H^k(101)$ for $k \ge 0$. 
There are four ways in which a word $w$ can be extended to $awb$, \ie with a symbol both to the right and left.
It is worth noting that for $w = \tau_k$ or $\bar \tau_k$, all four ways indeed occur in $\K$, while for $w = \tau_k \bar \tau_k \tau_k$ or 
$\bar \tau_k \tau_k \bar \tau_k$ only two extensions occur.
\item 
The Thue-Morse sequence has low {\em word-complexity}:
$$
p(n) = \begin{cases}
3 \cdot 2^m + 4r & \text{ if } 0 \le r < 2^{m-1},\\
4 \cdot 2^m + 2r & \text{ if } 2^{m-1} \le  r < 2^m,\\
\end{cases}
$$
where $n = 2m + r + 1$. 
\item The Thue-Morse shift
is almost square-free in the sense that if $w = w_1\dots w_n$
is some word, then $ww$ can appear in $\K$, but not $www_1$.
The nature of the Thue-Morse substitution is such that
$\rho_0$ and $\rho_1$ are concatenations of the words $\tau_k$ and $\bar \tau_k$.
Appearances of $\tau_k$ and $\bar \tau_k$ can overlap, but not
for too long compared to their lengths, as made clear in 
Corollary~\ref{cor-limited-overlap}
\end{enumerate}

The next lemma shows that almost-invertibility of $\s$ on $\K$ implies some shadowing close to $\K$. 

\begin{lemma}\label{lem-invertible}
For $x \in \S$ with $d(x,\K) < 2^{-5}$, 
let  $y, y' \in \K$ be the closest points in $\K$
to $x$ and $\sigma(x)$ respectively.
If $y' \neq \s(y)$, then $y'$ starts as
$\tau_k$, $\bar \tau_k$, $\tau_k\bar \tau_k \tau_k$ or
$\bar \tau_k \tau_k \bar \tau_k$ for some $k \ge 3$.
\end{lemma}

\begin{proof}
As $y' \neq \s(y)$, there is another $z \in \K$ such that $\s(z) = y'$
and $z_0 \neq y_0 = x_0$.
Let $d$ be maximal such that $y_1\dots y_{d-1} = z_1 \dots z_{d-1}$,
so $y_d \neq z_d$.
This means that the word $y_1\dots y_{d-1}$ is bi-special, and according to property (3) has to coincide with $\tau_k$, $\bar \tau_k$, $\tau_k\bar \tau_k \tau_k$ or
$\bar \tau_k \tau_k \bar \tau_k$ for some $k \ge 3$.
\end{proof}

Due to the Cantor structure of $\K$, the distance of an orbit to $\K$ 
is not a monotone function in the time. This is the main problem we will have to deal with. 

\begin{definition}
\label{def-accident}
Let $x \in \S$ be such that $d(\s(x),\K)<2d(x,\K)$. Then we say that we have an  \emph{accident} at $\s(x)$. 
By extension, if $d(\s^{k+1}(x),\K)=2d(\s^{k}(x),\K)$ for every $k<n-1$, 
but $d(\s^{n}(x),\K)<2d(\s^{n-1}(x),\K)$, then we say that we have an \emph{accident} at time $n$. 
\end{definition}

\begin{proposition}\label{prop-time-accident}
Assume that $-\log_{2}d(x,\K)=d$ and that $b\le d$ is the first accident for the piece of orbit $x,\ldots, \s^{d}(x)$, then 
\begin{itemize}
\item $x_{b}x_{b+1}\ldots x_{d-1}$ is a bispecial word for $\K$; \vspace{1mm}
\item $d-b = 3^{\eps} \cdot 2^k$ for some $k$ and $\eps\in\{0,1\}$;
\item $x_0\dots x_{d-1}$ is neither right-special nor left-special; \vspace{2mm}
\item $b \ge \ \begin{cases}
2^k & \text{ if } d-b = 2^k, \\
2^{k+1}  & \text{ if } d-b = 3 \cdot 2^k.
\end{cases}
$
\end{itemize}
\end{proposition}

\begin{proof}
Let $y$  and $y' \in \K$ be such that $x$ and $y$ coincide for $d$ digits 
and $\s^{b}(x)$ and $y'$ coincide for at least $d-b$ digits. Then
$$
x_{b}x_{b+1}\ldots x_{d-1} =
y_{b}y_{b+1}\ldots y_{d-1} = y'_{0}y'_{1}\ldots y'_{d-b-1}
$$
is a right-special because it can be continued both as $y'$ and $y$. 
The word $x_{b}x_{b+1}\ldots x_{d-1}$ is also a left-special, because otherwise, by Lemma~\ref{lem-invertible}, only one preimage of $y'$ by $\s$ would be in $\K$ and this would coincide with the word $y_{b-1}y_{b}\ldots y_{d-b-1}$. Then $b-1$ rather than $b$ would be the first accident.  By property (3) above, $d-b = 3^{\eps}2^{k}$.  
On the other hand, $x_0\dots x_{d-1}$ cannot be right-special, because otherwise there would be a point
$\tilde x = x_0\dots x_{d-1}\bar x_d\dots \in \K$ with $d(x, \tilde x) < 2^{-d}$.
If $x_0\dots x_{d-1}$ is left-special, then 

To finish the proof of the proposition we need to check that the next accident cannot happen too early. 
Assume that $x_0\dots x_{d-2}$ start as $\rho_0 = r_0r_1r_3\dots$ (the argument for $\rho_1$ is the same).
Let $\pi(n) = \#\{ 0 \le i < n : r_i = 1\} -  \#\{ 0 \le i < n : r_i = 0\}$
count the surplus of $1$'s within the first $n$ entries of $\rho_1$.
Clearly $\pi(n) = 0$ for even $n$ and $\pi(n) = \pm 1$ otherwise.
Assume the word $\tau_k$ starts in $\rho_0$ at some digit $m < 2^k$. 
If $\pi(m) = 1$, then $\pi(m+3) = 2$ while if $\pi(m) = -1$, then
$\pi(m+7) = -2$. A similar argument works if $\bar \tau_k$ stars at
digit $m$. This shows that if $\tau_k$ or $\bar \tau_k$ can only start
in $\rho_0$ at even digits.
This means that we can take the inverse $H^{-1}$ and find
that $\tau_{k-1}$ (or $\bar \tau_k$) start at digit $m/2 < 2^{k-2}$ in $\rho_0$.
Repeating this argument, we arrive at $\tau_3$ or $\bar \tau_3$ starting
before digit $8 = 2^{4-1}$ of $\rho_0$, which is definitely false,
as we can see by inspecting $\rho_0 = 0110\ 1001\ 1001\ 0110\ \dots$.
Note also that the bound $2^k$ is sharp,
because $\bar \tau_k$ starts in $\rho_0$ at entry $2^k$.

Finally, we need to answer the same question for bispecial words $\tau_k \bar \tau_k \tau_k = \tau_{k+1} \tau_k$ and $\bar \tau_k \tau_k \bar \tau_k = \bar \tau_{k+1} \bar \tau_k$.
The previous argument shows that neither can start before digit $2^{k+1}$,
and also this bound is sharp, because $\bar \tau_k \tau_k \bar \tau_k$
starts in $\rho_0$ at entry $2^{k+1}$.
\end{proof}

\begin{corollary}\label{cor-limited-overlap}
Occurrences of $\tau_k$ and $\bar \tau_k$ cannot
overlap for more than $2^{k-1}$ digits.
\end{corollary}

\begin{proof}
We consider the prefix $\tau_k$ of $\rho_0$ only, as the other case is 
symmetric.
If the overlap was more than $2^{k-1}$ digits, then $\tau_{k-1}$ or 
$\bar \tau_{k-1}$ would appear in $\rho_0$ before digit $2^{k-1}$, which contradicts part (3) of Proposition~\ref{prop-time-accident}
\end{proof}

\begin{lemma}\label{lem:disjoint}
For each $k \ge 1$, the Thue-Morse substitution $H$ satisfies
$\K = \bigsqcup_{j=0}^{2^k-1} \s^j \circ H^k(\K)$,
where $\sqcup$ indicates disjoint union, so
$\s^i \circ H^k(\K) \cap \s^j \circ H^k(\K) = \emptyset$ for all
$0 \le i < j < 2^k$.
\end{lemma}
\begin{proof}
Take $x \in \K$, so there is a sequence $(n_k)_{k \in \N}$ such that
$x = \lim_k \s^{n_k}(\rho_0)$. If this sequence contains infinitely many even integers, then 
$x = \lim_k \s^{2m_k}(\rho_0) = \lim_k \s^{2m_k} \circ H(\rho_0)
= \lim_k H \circ  \s^{m_k}(\rho_0) \in H(\K)$. Otherwise, 
 $(n_k)_{k \in \N}$ contains infinitely many odd integers and
$x = \lim_k \s^{1+2m_k}(\rho_0) = \lim_k \s \circ \s^{2m_k} \circ H(\rho_0)
= \lim_k \s \circ  H \circ  \s^{m_k}(\rho_0) \in 
\s \circ H(\K)$.
Therefore
$\K \subset H(\K) \cup \s \circ H(\K)$.

Now if $x = H(a) = \s \circ H(b) \in \K$, then 
$$
x =  a_{0}\bar a_0 a_{1}\bar a_1 a_{2}\bar a_2\ldots
=  \bar b_0 b_{1}\bar b_1 b_{2}\bar b_2\ldots\ ,
$$
so $\bar b_0 = a_0 \neq \bar a_0 = b_1 \neq \bar b_1 = a_1 \neq \bar a_1 = b_2 \neq b_2 = a_2$. Therefore
$x = 101010\dots$ or $010101\dots$, but neither belongs to $\K$.

Now for the induction step, assume 
$\K = \bigsqcup_{j=0}^{2^k-1} \s^j \circ H^k(\K)$.
Then since $H$ is one-to-one,
\begin{eqnarray*}
\K &=& \bigsqcup_{j=0}^{2^k-1} \s^j \circ H^k(H(\K) \sqcup \s \circ H(\K)) \\
&=&  \left(\bigsqcup_{j=0}^{2^k-1} \s^j \circ H^{k+1}(\K) \right)
\bigsqcup \left( \bigsqcup_{j=0}^{2^{k-1}} \s^j \circ H^k \circ \s \circ H(\K)) \right)\\
&=&  \left(\bigsqcup_{j=0}^{2^k-1} \s^j \circ H^{k+1}(H(\K) \right)
\bigsqcup \left( \bigsqcup_{j=0}^{2^k-1} \s^{j+2^k} \circ H^{k+1}(\K)) \right)\\
&=&  \bigsqcup_{j=0}^{2^{k+1}-1} \s^j \circ H^{k+1}(\K).
\end{eqnarray*}
\end{proof}

\begin{lemma}
\label{lem-sigmacirch}
Let $x$ be in the cylinder $[ab]$ with $a, b \in \{0,1\}$. 
Then the accumulation point of $(\s\circ H)^k(x)$ are $0\rho_b$ and 
$1\rho_b$. 
More precisely, the $(\s\circ H)^{2k}(x)$ converges to $a\rho_{b}$ 
and $(\s\circ H)^{2k+1}(x)$ converges to $\bar a\rho_{b}$.
\end{lemma}
\begin{proof}
By definition of $H$ we get $H(x)= a \bar a H(b)\ldots$ 
Hence $\s\circ H(x)=\bar a H(b)\ldots$ By induction we get 
$$
(\s\circ H)^{2k}(x)=a H^{2k}(b)\ldots
\quad \mbox{ and } \quad (\s\circ H)^{2k+1}(x)=\bar a H^{2k+1}(b).
$$
Therefore $H^{n}(b)$ converges to $\rho_{b}$, for $b=0,1$. 
\end{proof}

\subsection{Continuous fixed points of $\CR$ on $\K$: Proof of Theorem~\ref{theo-super-uc}}\label{subsec:fixed_on_K}

We recall that we have $\CR (V) = V \circ \s \circ H + V \circ H$.
Therefore
\begin{eqnarray*}
\CR^2V &=& \CR(V \circ \s \circ H + V \circ H) \\
&=& V \circ \s \circ H \circ \s \circ H + V \circ \s \circ H^2 +
V \circ H \circ \s \circ H + V \circ H^2 \\
&=& V \circ \s^3 \circ H^2 + V \circ \s^2 \circ H^2 +
V \circ \s \circ H^2 + V \circ H^2, 
\end{eqnarray*}
and in general
$$
\CR^nV = S_{2^n}V \circ H^n \quad \text{ where } \quad(S_kV)(x) = \sum_{i=0}^{k-1} V \circ \s^i(x)
$$
is the $k$-th ergodic sum.

\begin{lemma}\label{lem-integralV}
If $V \in L^1(\mu_\K)$ is a fixed point of $\CR$, then $\int_{\K} V \ d\mu_{\K} = 0$.
\end{lemma}

\begin{proof}
For any typical (w.r.t.\ Birkhoff's Ergodic Theorem) $y \in \K$ we get 
$$
V(y)=(\CR^{n}V)(y)=\sum_{j=0}^{2^{n}-1}V\circ \s^j\circ H^{n}(y).
$$ 
Hence 
$$
\frac1{2^{n}}V(y)=\frac1{2^{n}}\sum_{j=0}^{2^{n}-1}V\circ \s^j\circ H^{n}(y).
$$
The left hand side tend to $0$ as $n \to \8$ and the right hand side 
tends to $\int_{\K}V\,d\mu_{\K}$.
 \end{proof}

\begin{lemma}\label{lem-Vpreimage-rhoi}
Let $W$ be any continuous fixed point for $\CR$ (on $\K$). Then, 
for $j=0,1$,
$$
W(01\rho_{j})+W(10\rho_{j})=0\quad \mbox{ and }\quad W(1\rho_{j})=
W(10\rho_{j})+W(0\rho_{j}).
$$
\end{lemma}
\begin{proof}
Using the equality $W(x)=(\CR W)(x)=W\circ H(x)+W\circ\s\circ H(x)$ we immediately get:
$$W\circ (\s\circ H)^n(x)=W\circ H\circ(\s\circ H)^n(x)+W\circ (\s\circ H)^{n+1}(x).$$
Using Lemma~\ref{lem-sigmacirch} on this new equality, we obtain 
$$
W(i\rho_{j})=W(i\bar{i}\rho_{j})+W(\bar{i}\rho_{j}),
$$
for $i, j \in \{0,1\}$. This gives the second equality of the lemma 
(for $i=1$). The symmetric formula is obtained from the case $i=0$,
and then adding both formulas yields $W(01\rho_{j})+W(10\rho_{j})=0$.
\end{proof}

\begin{remark}\label{rem-continuity-weaker}
Lemma~\ref{lem-Vpreimage-rhoi} still holds if the potential is only 
continuous at points of the form $i\rho_{j}$ and $i\bar i\rho_{j}$ 
with $i,j\in\{0,1\}$. $\hfill \blacksquare$
\end{remark}

Recall the one-parameter family of potentials $U_c$ from \eqref{eq:Uc}.
They are fixed points of $\CR$, not just on $\K$, but globally on $\S$.
Let $i:\S \to \S$ be the involution changing digits $0$ to $1$ 
and vice versa. Clearly $U_c = -U_c \circ i$.
We can now prove Theorem~\ref{theo-super-uc}.

\begin{proof}[Proof of Theorem~\ref{theo-super-uc}] 
Let $W$ be a potential on $\K$, that is fixed by $\CR$. We assume that the 
variations are summable:
$\sum_{k=1}^{\8}\Var_{k}(W)<\8$. 

We show that $W$ is constant on $2$-cylinders. Let $x=x_{0}x_{1}\ldots$ and $y=y_{0}y_{1}\ldots$ be in the same $2$-cylinder (namely $x_{0}=y_{0}$ and $x_{1}=y_{1}$). Then, for every $n$, $H^n(x)$ and $H^n(y)$ coincide for (at least) $2^{n+1}$ digits. Therefore  
\begin{eqnarray*}
|W(x)-W(y)|&=& |(\CR^{n}W)(x)-(\CR^{n}W)(y)|\\
&=& |(S_{2^{n}}W)(H^n(x))-(S_{2^n}W)(H^n(y))|\\
&\le& \sum_{k=2^n+1}^{2^{n+1}}\Var_{k}(W).
\end{eqnarray*}
Convergence of the series $\sum_{k}\Var_{k}(W)$ implies that  $\sum_{k=2^n+1}^{2^{n+1}}\Var_{k}(W) \to 0$ as $n \to \8$. 
This yields that $W$ is constant on $2$-cylinders. 

Lemma~\ref{lem-Vpreimage-rhoi} shows that
$W|_{[01]}=-W|_{[10]}$. Again, the second equality in
that lemma used for both $\rho_{0}$ and $\rho_{1}$ shows that
$W|_{[00]}=W|_{[11]}=0$. 
Therefore $W = U_c$ with $c = W(\rho_0)$, and
the proof is finished. 
\end{proof}

\subsection{Global fixed points for $\CR$: Proof of Theorem~\ref{theo-vtilde}}\label{subsec-theo-vtilde}

To give an idea why Theorem~\ref{theo-vtilde} holds, observe that
the property $V(x) = \frac1n + o(\frac1n)$ if $d(x, \K) = 2^{-n}$
(so $V$ vanishes on $\K$ but is positive elsewhere) is in spirit preserved under
iterations of $\CR$, provided the shift $\s$ doubles the distance 
from $\K$. Let  $\CH$ denote the class of potentials satisfying this property.
Choose $x$ such that $d(x,\K) = 2^{-m}$.
Taking the limit of Riemann sums, and since $\CR$ preserves 
the class of non-negative functions, we obtain
\begin{eqnarray*}
0 \le (\CR^nV)(x) &=& \sum_{j=0}^{2^n-1} \frac{1}{2^nm-j} +
\sum_{j=0}^{2^n-1} o(\frac{1}{2^nm-j})  \\
&\rightarrow_{\ninf}& (1+o(1)) \int_0^1 \frac{1}{m-t} \ dt \\
&=& (1+o(1)) \log \frac{m}{m-1} = 
\frac1m + o(\frac1m).
\end{eqnarray*}
However, it may happen that $d(\s(y), \K) < 2d(y, \K)$ for some
$y = \s^j \circ H^n(x)$, in which case we speak of an accident (see Definition~\ref{def-accident}).
The proof of the proposition includes
an argument that accidents happen only infrequently,
and far apart from each other.

\begin{remark}\label{rem-o1m}
We emphasize an important bi-product of the previous computation. If $V$ is of the form $V(x)=o(\frac1m)$ when $d(x,\K)=2^{-m}$, then $\CR^{n}(V)$ converges to 0. See also Proposition~\ref{prop-renorm-power}.
$\hfill \blacksquare$
\end{remark}

\begin{proof}[Proof of Theorem~\ref{theo-vtilde}]
The proof has three steps. In the first step we prove that the class $\CH$ is invariant under $\CR$. In the second step we show that  $\CR^{n}(V_{0})$, with $V_{0}$ defined by $V_0(x) = \frac1m$ if $d(x,\K) = 2^{-m}$, is positive (outside $\K$) and bounded from above. In the last step we deduce from the two first steps that there exists a unique fixed point and that it is continuous and positive. We also briefly explain why it gives the result for any $V \in \CH$. 

\medskip
\noindent
\emph{Step 1.}
We recall that $\CR$ is defined by 
$(\CR V)(x):=V\circ H(x)+V\circ\s\circ H(x)$.
As $H$ and $\s$ are continuous, $\CR(V)$ is continuous if $V$ is continuous. Let $x\in \S$,
then if $x_{K}\in\K$ is such that 
\begin{equation}\label{eq:double} 
d(x,\K)=d(x,x_{K})=2^{-m}, \text{ then } d(H(x), H(x_K)) = 2^{-2m}.
\end{equation}
We claim that if $m \ge 3$, then $d(H(x),\K) = d(H(x),H(x_{K}))$.
Let us assume by contradiction that $y\in \K$ is such that 
$d(H(x),\K)=d(H(x),y)<d(H(x),H(x_{K}))$. 
By Lemma~\ref{lem:disjoint}, $y$ belongs either to $H(\K)$ or to $\s\circ H(\K)$. In the first case, say $H(z)=y$, we get 
$$
d(H(x),H(z)) < d(H(x),H(x_{K})).
$$
This would yield $d(x,z)<d(x,x_{K})$ which contradicts the fact that $d(x,\K) = d(x,x_{K})$. 

In the other case, say $y=\s\circ H(z)$, $m\ge 3$ yields $H(x) = a_{0}\bar a_0a_{1}\bar a_1a_{2}\bar a_2\ldots$ and  $\s\circ H(z) = \bar b_0 b_{1}\bar b_1b_{2}\bar b_2\ldots$. As in the proof of Lemma~\ref{lem:disjoint} this would show that  $y$ must start with $010101$
 or $101010$. However, both are forbidden in $\K$ and this produces a contradiction.  This finishes the proof of the claim. 

 Lemma~\ref{lem-invertible} also shows that 
$d(\s\circ H(x),\K)=2^{-(2m-1)}=d(\s\circ H(x),\s\circ H(x_{K}))$. Therefore 
\begin{equation}\label{eq_R}
(\CR V)(x)=V\circ H(x)+V\circ \s\circ H(x)=
\frac1{{2m}}+\frac1{{2m-1}}+o(\frac1{{m}})=\frac1{m}+o(\frac1m).
\end{equation}

\medskip
\noindent
\emph{Step 2.} We establish upper and lower bounds for $\CR^n(V_0)$
where $V_{0}$ is defined by $V_0(x) = \frac1m$ if $d(x,\K) = 2^{-m}$. Let $x \in \S$ be such that $d(x,\K) = 2^{-m}$, and pick $x_K \in \K$ 
such that $x$ and $x_K$ coincide for exactly $m$ initial digits. Due to the definition of $\K$, $m \ge 2$ (for any $x$) 
but we assume in the following that $m \ge 3$. By \eqref{eq:double} 
we have $d(H^nx, \K) = d(H^nx, H^nx_K) = 2^{-2^nm}$.
Assume that the first digit of $x$ is $0$. Then $H^{n}(x)$ coincides with $\rho_{0}$ at least for $2^{n}$ digits.

 Assume now that $H^nx$ has an accident at the $j$-th shift, $1\le j < 2^n$,
so there is $y \in \K$ such that $d(\s^j \circ H_n(x), y) <2
d(\s^j \circ H_n(x), \s^j \circ H_n(x_K))$.

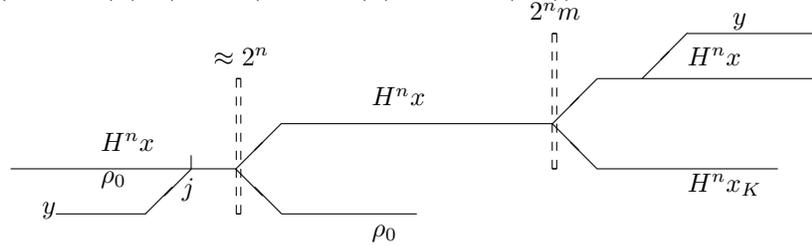
\begin{figure}[ht]
\unitlength=6mm
\begin{picture}(18,5)(0,0)
\put(0,2){\line(1,0){5}}
  \put(2,1.7){\small $\rho_0$}  \put(2,2.4){\small $H^nx$} 
\put(5,2){\line(1,1){1}} \put(5,2){\line(1,-1){1}}
\put(4,2){\line(0,1){0.3}}\put(3.8,1.4){\small $j$}
\put(4,2){\line(-1,-1){1}}\put(3,1){\line(-1,0){2}}\put(0.7,1){\small $y$}
\put(5,1){\dashbox{0.2}(0.1,3)}\put(4.5,4.3){\small $\approx 2^n$} 
\put(6,3){\line(1,0){6}}
  \put(8,3.4){\small $H^nx$} 
\put(6,1){\line(1,0){3}}
  \put(8,0.5){\small $\rho_0$} 
\put(12,3){\line(1,1){1}} \put(12,3){\line(1,-1){1}}
\put(12,2){\dashbox{0.2}(0.1,3)}\put(11.5,5.3){\small $2^nm$} 
\put(13,4){\line(1,0){5}}
  \put(15,4.3){\small $H^nx$} 
  \put(14, 4){\line(1,1){1}}\put(15,5){\line(1,0){3}}\put(16,5.2){\small $y$}
\put(13,2){\line(1,0){4}}
  \put(15,1.5){\small $H^nx_K$} 
\end{picture}
\caption{Half of the sum $\CR^nV$ can easily be estimated.}\label{fig:compare}
\end{figure}

The last point in Proposition~\ref{prop-time-accident} shows $j\ge 2^{n-1}$. 
Therefore, using again that the sum approximates the Riemann integral,
\begin{eqnarray*}
(\CR^nV_0)(x) \ge \frac{1}{2^n}  \sum_{j=0}^{(2^n/2)-1} \frac{1}{m-j/2^n} 
\to_{\ninf} \int_0^\frac12 \frac{1}{m-x} \ dx \ge \frac1{2m}. 
\end{eqnarray*}
The worst case scenario for the upper bound is when there is no accident,
and then 
\begin{equation}
\label{equ-riemansum-upper}
(\CR^nV_{0})(x) = \sum_{j=0}^{2^n-1} \frac{1}{2^nm-j} \to_{\ninf} \int_0^1 \frac{1}{m-x} \ dx \le \frac1{m-1}
\end{equation} 
as required.

\begin{remark}\label{rem-pointproche}
Note that the largest distance between  $\K$ and points $\s^{k}(H^n(x))$ 
with $k\in \llb0,2^n-1\rrb$ is smaller than 
$2^{-(2^nm-2^{n}+1)}\le 2^{-2^n}$. This largest distance thus  tends 
to $0$ super-exponentially fast as $n \to \8$. 
$\hfill \blacksquare$\end{remark}

\medskip
\noindent
\emph{Step 3.}
We prove here  equicontinuity for $\CR^{n}(V_{0})$. Namely, there exists some positive $\kappa$, such that for every $n$, for every $x$ and $y$
$$
|\CR^{n}(V_{0})(x)-\CR^{n}(V_{0})(y)|\le \frac\kappa{|\log_2 d(x,y)|},
$$
holds. 

Assume that $x$ and $y \in \S$ coincide for $m$ digits. 
We consider two cases. 

{\bf Case 1:} $d(x,\K)=2^{-m'}=:d(x,z)$, with $m'<m$ (and $z\in \K$). 

\begin{figure}[ht]
\unitlength=6mm
\begin{picture}(18,5.5)(0,0)
\put(0,2){\line(1,0){5}}
\put(5,2){\line(1,1){1}} \put(5,2){\line(1,-1){1}}
\put(5,1){\dashbox{0.2}(0.1,3)}\put(4.5,4.3){\small $ 2^nm'$} 
\put(6,3){\line(1,0){6}}
\put(6,1){\line(1,0){3}}
  \put(8,0.5){\small $H^n(z)$} 
\put(12,3){\line(1,1){1}} \put(12,3){\line(1,-1){1}}
\put(12,2){\dashbox{0.2}(0.1,3)}\put(11.5,5.3){\small $2^nm$} 
\put(13,4){\line(1,0){5}}
  \put(15,4.4){\small $H^n(x)$} 
\put(13,2){\line(1,0){4}}
  \put(15,1.5){\small $H^n(y)$} 
\end{picture}
\end{figure}

If there are no accidents for $\s^{j}\circ H^n(x)$ for $j\in \llb 0, 2^n\llb$, then for every $j$,
$$d(\s^j(H^n(x)),\K)=d(\s^j(H^n(y)),\K)=d(\s^j(H^n(x)),\s^j(H^n(z))),$$
and $V_{0}(\s^j(H^n(x)))=V_{0}(\s^j(H^n(y)))$. 
This yields $(\CR^n V_{0})(x)=(\CR^n V_{0})(y)$. 

{\bf Case 2:} 
If there is an accident, say at time $j_{0}$, then two sub-cases can happen. 

Subcase 2-1. The accident is due to a point $z'$ that separates before $2^nm$,
see Figure~\ref{fig:case2-1}. 
 
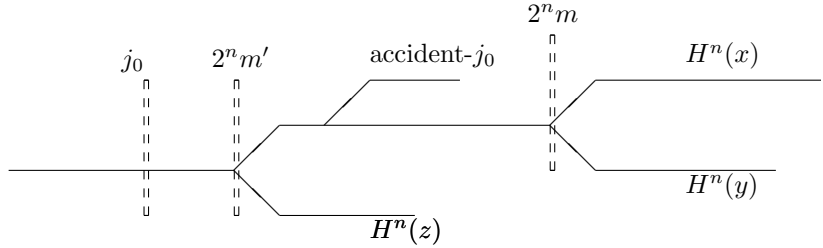
\begin{figure}[ht]
\unitlength=6mm
\begin{picture}(18,6)(0,0)
\put(0,2){\line(1,0){5}}
\put(5,2){\line(1,1){1}} \put(5,2){\line(1,-1){1}}
\put(3,1){\dashbox{0.2}(0.1,3)}\put(2.5,4.3){\small $ j_{0}$}
\put(5,1){\dashbox{0.2}(0.1,3)}\put(4.5,4.3){\small $ 2^nm'$} 
\put(6,3){\line(1,0){6}}
\put(6,1){\line(1,0){3}}
  \put(8,0.5){\small $H^n(z)$} 
\put(7,3){\line(1,1){1}}
 \put(8,4.4){\small accident-$j_{0}$} 
\put(8,4){\line(1,0){2}}
  \put(8,0.5){\small $H^n(z)$} 
\put(12,3){\line(1,1){1}} \put(12,3){\line(1,-1){1}}
\put(12,2){\dashbox{0.2}(0.1,3)}\put(11.5,5.3){\small $2^nm$} 
\put(13,4){\line(1,0){5}}
  \put(15,4.4){\small $H^n(x)$} 
\put(13,2){\line(1,0){4}}
  \put(15,1.5){\small $H^n(y)$} 
\end{picture}
\caption{Comparing sequence when the accident occurs before separation.}\label{fig:case2-1}
\end{figure}
Again, we claim that $V_{0}(\s^j(H^n(x)))=V_{0}(\s^j(H^n(y)))$ holds for $j\le j_{0}-1$, but also for $j \ge j_{0}$ but smaller than the (potential) second accident. Going further, we refer to cases 2-2 or 1. 

Sub-case 2-2. The accident is due to a point much closer to $H^{n}(x)$ than to $H^{n}(y)$, see Figure~\ref{fig:case2-2}.

\begin{figure}[ht]
\unitlength=6mm
\begin{picture}(18,5.5)(0,0)
\put(0,2){\line(1,0){5}}
\put(5,2){\line(1,1){1}} \put(5,2){\line(1,-1){1}}
\put(3,1){\dashbox{0.2}(0.1,3)}\put(2.5,4.3){\small $ j_{0}$}
\put(5,1){\dashbox{0.2}(0.1,3)}\put(4.5,4.3){\small $ 2^nm'$} 
\put(6,3){\line(1,0){6}}
\put(6,1){\line(1,0){3}}
  \put(8,0.5){\small $H^n(z)$} 
  \put(8,0.5){\small $H^n(z)$} 
\put(12,3){\line(1,1){1}} \put(12,3){\line(1,-1){1}}
\put(12,2){\dashbox{0.2}(0.1,3)}\put(11.5,5.3){\small $2^nm$} 
\put(13,4){\line(1,0){5}}
  \put(15,4.2){\small $H^n(x)$} 
  \put(13.5,4){\line(1,1){1}}
 \put(14.5,5.2){\small accident-$j_{0}$} 
\put(14.5,5){\line(1,0){2}}

\put(13,2){\line(1,0){4}}
  \put(15,1.5){\small $H^n(y)$} 
\end{picture}
\caption{Comparing sequence when the accident occurs after separation.}\label{fig:case2-2}
\end{figure}
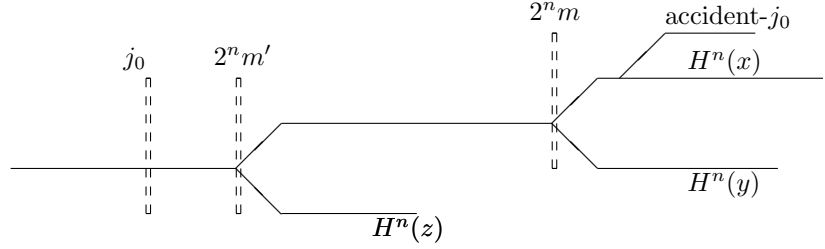
In that case we recall that the first accident cannot happen before $2^{n-1}$, hence $j_{0}\ge 2^{n-1}$. 
Again, for $j\le j_{0}-1$ we get $V_{0}(\s^j(H^n(x)))=V_{0}(\s^j(H^n(y)))$.
By definition of accident we get
$$\max\left\{
V_{0}(\s^{j+2^{n-1}}(H^n(x)))\ , \ V_{0}(\s^{j+2^{n-1}}(H^n(y)))\right\}
\le \frac1{2^nm-2^{n-1}-j}
$$
for $j\ge j_{0}$. This yields 
$$
\left| (\CR^{n} V_{0})(x)-(\CR^n V_{0})(y)\right|
\le \sum_{k=j}^{2^{n-1}}\frac2{2^nm-2^{n-1}-j}=\frac1{2^n}\sum_{k=j}^{2^{n-1}}\frac2{m-\frac12-\frac{j}{2^n}}.
$$ 
This last sum is a Riemann sum and is thus (uniformly in $n$) comparable to the associated integral
$\int_{0}^{\frac12}\frac1{m-\frac12-t}dt
\le \frac{1}{2(m-1)}.$

\medskip
\noindent
\emph{Step 4.}
Following Step 3, the family $(\frac1n\sum_{k=0}^{n-1}\CR^k(V_{0}))_n$ is equicontinuous (and bounded), hence there exists accumulation points. 
Let us prove that $\disp(\frac1n\sum_{k=0}^{n-1}\CR^k(V_{0}))_n$ actually converges. 

Assume that $\wt V_{1}$ and $\wt V_{2}$ are two accumulation points. Note that both $\wt V_{1}$ and $\wt V_{2}$ are  fixed points for $\CR$. 
 They are continuous functions and Steps 1 and 2 show that they satisfy 
$$\frac1{2m}\le \wt V_{i}(x)\le \frac1{m}+o(\frac1m),$$
if $d(x,\K)=2^{-m}$. From this we get 
$$
\wt V_{1}(x)-\wt V_{2}(x)\le \frac1{2m}+o(\frac1m)=\frac12V_{0}(x)+o(V_{0}(x)),
$$
and then for every $n$, 
$$
\wt V_{1}-\wt V_{2}=\frac1n\sum_{k=0}^{n-1}\CR^{k}(\wt V_{1})-\CR^{k}(\wt V_{2})\le\frac12\frac1n\sum_{k=0}^{n-1}\CR^{k}(V_{0})+o(\CR^{k}(V_{0})).
$$
We recall from Remark~\ref{rem-o1m} that $o(\CR^{k}(V_{0}))$ goes to 0 as $k\to\8$. 
Taking the limit on the right hand side along the subsequence which converges to $\wt V_{2}$ we get 
$$
\wt V_{1}-\wt V_{2}\le \frac12\wt V_{2} + o(V_0),
$$
which is equivalent to $\frac23\wt V_{1}\le \wt V_{2} + o(V_0)$. Exchanging $\wt V_{1}$ and $\wt V_{2}$ we also get $\frac23\wt V_{2}\le \wt V_{1} + o(V_0)$.
These two inequalities yield 
$$
\wt V_{1}-\wt V_{2}\le \frac13 V_{0}+o(V_{0})
\quad \text { and } \quad
\wt V_{2}-\wt V_{1}\le \frac13 V_{0}+o(V_{0}).
$$
Again, applying $\CR^{k}$ on these inequalities and the Cesaro mean, we get 
$$
\wt V_{1}-\wt V_{2} \le \frac13\wt V_{2} + o(V_0) 
\quad \text{ and } \quad 
\wt V_{2}-\wt V_{1}\le \frac13\wt V_{1} + o(V_0).
$$
Iterating this process, we get that for every integer $p$, 
$$\frac{p}{p+1}\wt V_{2} + o(V_0) \le \wt V_{1}\le \frac{p+1}p\wt V_{2} + o(V_0).$$
This proves $\wt V_{1} - \wt V_{2} = o(V_0)$, \ie
$(\wt V_{1} - \wt V_{2})(x) \to 0$ faster than $V_0(x)$ as $x \to \K$ (see again Remark~\ref{rem-o1m}).
But  $\wt V_{1} - \wt V_{2}$
is also fixed by $\CR$, so we can apply \eqref{equ-riemansum-upper}
with a factor $o(V_0)$ in front. This shows that  $\wt V_{1} = \wt V_{2}$,
and hence the convergence of the Cesaro mean 
$(\frac1n\sum_{k=0}^{n-1}\CR^{k}(V_{0}))_n$. 
This finishes the proof of Theorem~\ref{theo-vtilde}.
\end{proof}

\subsection{More results on fixed points of $\CR$}

The same proof also proves a more general result:
\begin{proposition}
\label{prop-renorm-power}
Let $a$ be a real positive number. Take $V(x) = \frac1{n^{a}}+o(\frac1{n^{a}})$ if $d(x,\K) = 2^{-n}$.
Then, for $a>1$, $ \lim_{n \to \infty} \CR^nV\equiv 0$ and for $a<1$, 
$\lim_{n \to \infty} \CR^nV\equiv\8$.
\end{proposition}
\begin{proof}
Immediate, since the Riemann sum as in \eqref{equ-riemansum-upper}
has a factor $2^{n(1-a)}$ in front of it. 
\end{proof}
Consequently, any $V$ satisfying
$V(x) = \frac1{n}+o(\frac1{n})\mbox{ for }d(x,\K) = 2^{-n}$
belongs to the weak stable set  $V\in \CW^{s}(\wt V)$ of the fixed 
potential $\wt V$ from Theorem~\ref{theo-vtilde}.
However, $\CW^{s}(\wt V)$ is in fact much larger:

\begin{proposition}
If $V(x) = \frac1n g(x)$ for $d(x, \K) = 2^{-n}$ and $g:\S \to \R$ a continuous function, then $\frac1j \sum_{k=0}^{j-1} \CR^k(V) \to \wt V \cdot \int_{\K} g\ d\mu_{\K}$.
\end{proposition}

\begin{proof}
Take $\eps > 0$ arbitrary, and take $r \in \N$ so large that
$\sup |g| 2^{-r} \le \eps$ and
if $d(x, \K) = d(x, x_{\K}) \le 2^{-r}$, then $|g(x) - g(x_{\K})| \le \eps$.
Next take $k \in \N$ so large that if $k = r+s$, then 
$$
\left| \frac{1}{2^s}\sum_{i=0}^{2^s-1} g(\s^i(y)) - 
\int g\ d\mu_{\K} \right| \le \eps.
$$
uniformly over $y \in \K$.
Then we can estimate
\begin{eqnarray*}
(\CR^k V)(x) &=& \sum_{j=0}^{2^k-1} V \circ \s^j \circ H^k(x) \\
 &\le& \frac{1}{2^k} \sum_{j=0}^{2^k-1} \frac{1}{m-\frac{j}{2^k}} 
g \circ \s^j \circ H^k(x) \\
&=& \frac{1}{2^r} \sum_{t=0}^{2^r-1}\frac{1}{2^s} \sum_{i=0}^{2^s-1} 
\frac{1}{m-\frac{1}{2^k} (2^st+i)} 
g \circ \s^{2^st+i} \circ H^k(x) \\
&=& \frac{1}{2^r} \sum_{t=0}^{2^r-2}\frac{1}{2^s} \sum_{i=0}^{2^s-1} 
\left( \frac{1}{m-\frac{t}{2^r}} + O(2^{-r}) \right) \cdot \int_{\K} \left( g\ d\mu_{\K} + O(\eps) \right) \\
&& \ + \ \frac1{2^r} \frac{1}{2^s}\sum_{j=0}^{2^s-1} \frac{1}{m-\frac{1}{2^k} (2^k-2^s+j)} 
\sup|g|\\
&\to& \int_0^1 \frac{1}{m-x} dx \cdot \int_{\K} g\ d\mu_{\K} + O(3 \eps).
\end{eqnarray*}
Since $\eps$ is arbitrary, we find
$\limsup_k (\CR^kV)(x) \le \frac1m \cdot \int_{\K} g\ d\mu_{\K} + o(\frac1m)$.
Similar to Step 2 in the proof of Theorem~\ref{theo-vtilde},
we find
$\limsup_k (\CR^kV)(x) \ge \frac1{2m} \cdot \int_{\K} g\ d\mu_{\K} 
+ o(\frac1m)$.
From this, using the argument of Step 3 in the proof of 
Theorem~\ref{theo-vtilde}, we conclude that for the Cesaro means,
$\lim_n \frac1n \sum_{k=0}^{n-1}(\CR^kV)(x) = 
\wt V(x) \cdot \int_{\K} g\ d\mu_{\K}$.
\end{proof}

\subsection{Unbounded fixed points of $\CR$}\label{subsec-unbounded1}

The application to Feigenbaum maps discussed in the Appendix of this paper suggests 
the existence of unbounded fixed points $V_u$ of $\CR$ as well.
They can actually be constructed explicitly using the
disjoint decomposition
$$
\S \setminus \s^{-1}\{ \rho_0, \rho_1 \} = \sqcup_{k \ge 0} \left( (\s \circ H)^k(\S) 
\setminus(\s \circ H)^{k+1}(\S) \right).
$$
If we set 
\begin{equation}\label{eq:Vu}
V_u|_{H(\S)} = g \quad \text{ and }\quad V_u(x) = V_u(y) - 
V_u \circ H(y)  \quad \text{ for } x = \s \circ H(y),
\end{equation}
then $V_u$ is well-defined and $\CR V_u = V_u$ on $\S \setminus \s^{-1}\{ \rho_0, \rho_1 \}$.
The simplest example is
\begin{equation}\label{eq_exampleVu}
V_u|_{ (\s \circ H)^k(\S) \setminus(\s \circ H^{k+1})(\S) } = (1-k)\alpha,
\end{equation}
and we will explore this further for phase transitions in 
Section~\ref{sec:thermo}.

For $x \in \S \setminus \s \circ H(\S)$ and 
$x^k = (\s \circ H)^k(x)$, we have
$$
V_u(x^k) = g(x) - \sum_{j=1}^k g \circ \s^{2^j-2} \circ H^j(x).
$$
Now for $x \in [1]$
$$
\s^{2^j-2} \circ H^j(x) \to \left\{ \begin{array}{ll}
\sigma^{-2}(\rho_0) & \text{ along odd $j$'s,}\\
\sigma^{-2}(\rho_1) & \text{ along even $j$'s,}
\end{array}\right. 
$$
and the reverse formula holds for $x \in [0]$.
In either case, $V(x^k) \sim \frac{k}{2} [g \circ \s^{-2}(\rho_0) + g \circ \s^{-2}(\rho_1)]$.
Therefore, unless
$g \circ \s^{-2}(\rho_0) + g \circ \s^{-2}(\rho_1) = 0$, 
the potential $V_u$ is unbounded near 
$\lim_{k\to\infty} (\s \circ H)^k(x) = \{\s^{-1}(\rho_0)\ , \ \s^{-1}(\rho_1)\}$, cf.\ Lemma~\ref{lem-sigmacirch}.

\begin{remark}
A variation of this stems from the decomposition
$$
\S \setminus \{ \rho_0, \rho_1 \} = \sqcup_{k \ge 0} \left(H^k(\S) \setminus H^{k+1}(\S) \right).
$$
In this case, if we define
$$
V'_u|_{\s \circ H(\S)} = g \quad \text{ and }\quad V'_u(x) = V'_u(y) - 
V'_u \circ \s(x)  \quad \text{ for } x = H(y),
$$
then $V'_u = \CR V'_u$ on $\S \setminus \{ \rho_0, \rho_1 \}$.
$\hfill \blacksquare$
\end{remark}

\section{Thermodynamic formalism}\label{sec:thermo}
In this section we prove Theorems~\ref{theo-thermo-a<1}, \ref{theo-super-MP} and \ref{theo-thermo-Vu}. 
In the first subsection we define an induced transfer operator as in \cite{leplaideur1} and use its properties. 
Then we prove both theorems.

\subsection{General results and a key proposition}

Let $V:\S \to \R$ be some potential function, and let $J$ be any cylinder 
such that on it, the distance to $\K$ is constant, say $\delta_{J}$.
Consider the first return map
$T:J \to J$, say with return time $\tau(x) = \min\{ n \ge 1 : \s^n(x) \in J\}$, so $T(x) = \s^{\tau(x)}(x)$. The sequence of successive return times is then denoted by $\tau^k(x)$, $k=1,2,\ldots$
The transfer operator is defined as
\begin{equation}\label{eq:L}
(\CL_{z, \gamma} g)(x) = \sum_{T(y) = x} e^{\Phi_{z,\gamma}(y)} g(y)
\end{equation}
where $\Phi_{z,\gamma}(y): = -\gamma (S_n V)(y) - n z$ if $\tau(y) = n$. For a given test function $g$ and a point $x\in J$,  $(\CL_{z, \gamma} g)(x)$ is thus a power series in $e^{-z}$. 

These operators extend the usual transfer operator. They were introduced in \cite{leplaideur1} and allow us to define \emph{local equilibrium states}, \ie equilibrium states for the potentials of the form $\Phi_{z,\gamma}$ and the dynamical system $(J,T)$. These local equilibrium states are later denoted by $\nu_{z,\gamma}$.

We emphasize that, using induction on $J$, these operators $\CL_{z, \gamma}$  
 allow us to construct equilibrium states for potentials which do not necessarily satisfy 
the Bowen condition (such as \eg the Hofbauer potential). 

Nevertheless, we need the following {\em local Bowen condition}:
there exists $C_V$ (possibly depending on $J$) such that 
\begin{equation}\label{eq:Bowen}
|(S_nV)(x)-(S_nV)(y)|\le C_{V},
\end{equation}
whenever $x, y \in J$ coincide for $n:=\tau^{k}(x)=\tau^{k}(y)$ indices. 
This holds, \eg if $V(x)$ depends only on the distance 
between $x$ and $\K$. 

\begin{lemma}\label{lem-zc}
Let $x \in J$ and let $\gamma$ and $z$ be such that $(\CL_{z,\gamma}\BBone_{J})(x) < \infty$. Then $(\CL_{z,\gamma}g)(y)<\infty$ for every $y \in J$  and for every continuous function $g:J\rightarrow\R$. 
\end{lemma}
\begin{proof}
Note that for any $x, y \in J$, 
$(\CL_{z,\gamma}\BBone_{J})(x)\approx e^{\pm C_{V}}(\CL_{z,\gamma}\BBone_{J})(y)$. Indeed, if $x'$ and $y'$ are two preimages of $x$ and $y$ in $J$, with the same return time $n$ and such that for every $k\in\llb 0,n\rrb$ $\s^{k}(x')$ and $\s^{k}(y')$ are in the same cylinder, then
$$
|(S_nV)(x')-(S_nV)(y')|\le C_{V}.
$$
Recall that $J$ is compact, and that every continuous function $g$ on $J$ is bounded. Hence convergence (\ie as power series) of 
$(\CL_{z,\gamma}\BBone_{J})(x)$ 
ensures uniform convergence over $y \in J$ for any continuous $g$. 
This finishes the proof of the lemma.
\end{proof}

For fixed $\gamma$, there is a critical $z_c$ such that $(\CL_{z, \gamma}\BBone_{J})(x)$
converges for all $z > z_c$ and $z_{c}$ is the smallest real number with this property. Lemma~\ref{lem-zc} shows that $z_{c}$ is independent of $x$.  The next result is straightforward. 

\begin{lemma}\label{lem-decreaz-zlambda}
The spectral radius $\lambda_{z,\gamma, }$ of $\CL_{z,\gamma}$ is decreasing in both $\gamma$ and $z$.
\end{lemma}

We are interested in the critical $z_c$ 
and the pressure $\CP(\gamma)$, both as function of $\gamma$. 
Both curves are decreasing (or at least non-increasing).
If the curve $\gamma \mapsto z_c(\gamma)$ avoids the horizontal axis, 
then there is no phase transition: 

\begin{proposition}\label{prop-equil-presspos}
Let $V$ be continuous and satisfying the local Bowen condition \eqref{eq:Bowen} for every cylinder $J$ disjoint and at constant distance from $\K$. Then
the following hold:
\begin{enumerate}
\item[1.] For every $\gamma\ge 0$, the critical $z_{c}(\gamma)\le \CP(\gamma)$.
\item[2.] Assume that the pressure $\CP(\gamma) > -\gamma\int V\,d\mu_{\K}$.
Then there exists a unique equilibrium state for $-\gamma V$ and 
it gives a positive mass to every open set in $\S$. 
Moreover $z_{c}(\gamma)<\CP(\gamma)$ and $\CP(\gamma)$ is analytic on the largest open interval where the assumption holds.
\item[3] If $(\CL_{z,\gamma}\BBone_{J})(\xi)$ diverges for every (or some) $\xi$ and for $z=z_{c}(\gamma)$, then $\CP(\gamma)>z_{c}(\gamma)$ and there is a unique equilibrium state for $-\gamma V$. 
\end{enumerate}
\end{proposition}

\begin{proof}
There necessarily exists an equilibrium state for $-\gamma V$. Indeed, the potential is continuous and the metric entropy is upper semi-continuous. Therefore any accumulation point as $\eps \to 0$ of a family of measures 
$\nu_{\eps}$ satisfying 
$$
h_{\nu_{\eps}}(\s)-\gamma\int V\, d\nu_{\eps}\ge \CP(\gamma)
$$
is an equilibrium state. 

The main argument in the study  of local equilibrium states as in \cite{leplaideur1} is that $z > z_{c}(\gamma)$ (to make the 
transfer operator ``converges'') and that $V$ satisfies the local Bowen property
 \eqref{eq:Bowen}. 
This property is used in several places and in particular, it yields for every $x$ and $y$ in $J$ and for every $n$:
$$
e^{-\gamma C_{V}}\le\frac{(\CL^{n}_{z\gamma}\BBone_{J})(x)}{(\CL^{n}_{z\gamma}\BBone_{J})(y)}\le e^{\gamma C_{V}}.
$$
To prove part 1., recall that  
$$
(\CL_{z,\gamma}\BBone_{J})(x):=\sum_{n=1}^{\8}\left(\sum_{x', T(x')=x,\tau(x)=n}e^{-\gamma (S_nV)(x')}\right)e^{-nz},
$$
which yields that $ z_c = \limsup_{n}\frac1n\log\left(\sum_{x', T(x')=x,\tau(x)=n}e^{-\gamma (S_nV)(x')}\right)$. 
To prove the inequality $z_{c}(\gamma)\le \CP(\gamma)$, 
we copy the proof of Proposition 3.10 in \cite{jacoisa}. Define the measure $\wt\nu$ as follows\label{preuve-def-nutilde-pgamme}: for $x$ in $J$ and for each $T$-preimage $y$ of $x$ there exists a unique $\tau(y)$-periodic point $\xi(y) \in J$, coinciding with $y$ until $\tau(y)$. 
Next we define the measure $\wt\nu_{n}$ as the probability measure proportional to 
$$
\sum_{\xi(y),\tau(y)=n}e^{\Phi_{\CP(\gamma),\gamma}(\xi(y))}\left(\sum_{j=0}^{n-1}\delta_{\s^{j}\xi(y)}\right)=\sum_{\xi(y),\tau(y)=n}e^{-\gamma (S_nV)(\xi(y))-n\CP(\gamma)}\left(\sum_{j=0}^{n-1}\delta_{\s^{j}\xi(y)}\right).
$$
The measure $\wt\nu$ is an accumulation point of $(\wt\nu_{n})_{n \in \N}$. 
It follows from the proof of \cite[ Lemma 20.2.3, page 264]{katok} that
\begin{equation}\label{equ-zc-pgamme-nutilde}
z_{c}(\gamma)\le h_{\wt\nu}(\s)-\gamma\int V\,d\wt\nu\le \CP(\gamma).
\end{equation}
\begin{remark}\label{rem-nutilde-J}
We emphasize that $\wt\nu_{n}(J)=\frac1n$ for each $n$, 
which shows that $\wt\nu(J)=0$.  
$\hfill \blacksquare$
\end{remark}

\medskip Now we prove part 2. 
 Let $\mu_{\gamma}$ be an  ergodic equilibrium state for $-\gamma V$.
The assumption $\CP(\gamma)>-\gamma\int V\,d\mu_{\K}$ means that the unique 
shift-invariant measure on $\K$ cannot be an equilibrium state (since $\s|_{\K}$ has zero entropy). 
Hence $\mu_{\gamma}$ gives positive mass to some cylinder $J$ in $\K^{c}$. 
Thus the conditional measure 
\begin{equation}
\label{eq-meas-open-out}
\nu_{\gamma}(\cdot):=\mu_{\gamma}(\cdot\cap J)/\mu_{\gamma}(J).
\end{equation}
is $T$-invariant (using the above notations). 

We now focus on the convergence (as power series) 
of $(\CL_{z,\gamma}\BBone_{J})(x)$ 
for any $x \in J$ and $z=\CP(\gamma)$. 
The inequality $z_{c}(\gamma)\le \CP(\gamma)$  does not ensure convergence of $(\CL_{z,\gamma} \BBone_{J})(x)$ for $z=\CP(\gamma)$. 
Again, we copy and adapt arguments from \cite[Proposition 3.10]{jacoisa} 
to get that  $(\CL_{z,\gamma}\BBone_{J})(x)$
converges and that 
the $\Phi_{z,\gamma}$-pressure is non-positive for $z=\CP(\gamma)$. 

In the case $z>\CP(\gamma)$, so $z > z_{c}(\gamma)$, we can 
apply the local thermodynamic formalism for $\Phi_{z,\gamma}$. Moreover $z>z_{c}(\gamma)$ means that $\frac{\partial}{\partial z}
(\CL_{z,\gamma}\BBone_{J})(x)$ converges. 
This implies by  \cite[Proposition 6.8]{leplaideur1}  
that there exists a unique equilibrium state $\nu_{z,\gamma}$ on $J$ for $T$ and for the potential $\Phi_{z,\gamma}$, 
and that the expectation $\int_J \tau \ d\nu_{z,\gamma} < \infty$. 
In other words, there exists a shift-invariant 
probability measure $\mu_{z,\gamma}$ such that 
$$
\mu_{z,\gamma}(J)>0, \mbox{ and } \nu_{z,\gamma}(\cdot):=\frac{\mu_{z,\gamma}(\cdot\cap J)}{\mu_{z,\gamma}(J)}.
$$
The equality $ h_{\nu_{z,\gamma}}(T)+\int\Phi_{z,\gamma}\,d\nu_{z,\gamma}=\log\lambda_{z,\gamma}$ (the spectral radius for $\CL_{z,\gamma}$) shows that
$$
h_{\mu_{z,\gamma}}(\s)-\gamma\int V\,d\mu_{z,\gamma}=z+\mu_{z,\gamma}(J)\log\lambda_{z,\gamma}.
$$
As $z > \CP(\gamma)$ we get $\lambda_{z,\gamma}\le 1$. Now the Bowen property of the potential shows that for every $x \in J$ and for every $n \ge1$:
$$(\CL^{n}_{z,\gamma}\BBone_{J})(x)=e^{\gamma C_V}\lambda^{n}_{z,\gamma}.$$
The power series is decreasing in $z$, thus the monotone Lebesgue convergence theorem shows that it converges for $z=\CP(\gamma)$. 
For this value of the parameter $z$, the spectral radius 
$\lambda_{\CP(\gamma),\gamma} \le 1$.
Following \cite{leplaideur1}, there exists a unique local equilibrium state, $\nu_{\CP(\gamma),\gamma}$ with pressure $\log\lambda_{\CP(\gamma),\gamma}\le 1$. This proves that the $\Phi_{\CP(\gamma),\gamma}$-pressure is non-positive.
  
\medskip
Now, we prove that  the $\Phi_{z,\gamma}$-pressure is non-negative for $z=\CP(\gamma)$. Indeed, by Abramov's formula 
\begin{eqnarray*}
0 &=& h_{\mu_{\gamma}}(\s)-\gamma\int V\,d\mu_{\gamma}- \CP(\gamma) \\
&=& \mu_{\gamma}(J)\left(h_{\nu_{\gamma}}(T)-\gamma\int (S_{\tau(x)}V)(x)\,d\nu_{\gamma}(x)-\CP(\gamma)\int \tau\,d\nu_{\gamma}\right),
\end{eqnarray*}
which yields 
$$
h_{\nu_{\gamma}}(T)-\gamma\int S_{\tau(x)}(V)(x)-\CP(\gamma) \tau(x)\,d\nu_{\gamma}(x)\ =\ 0.
$$

\bigskip
Finally, as the $\Phi_{\CP(\gamma),\gamma}$-pressure is non-negative and 
non-positive, it is  equal to $0$. It also has a unique equilibrium state which is a Gibbs measure (in $J$ and for the first-return map $T$). As the conditional measure $\nu_{\gamma}$ has zero $\Phi_{\CP(\gamma),\gamma}$-pressure, it is the unique local equilibrium state. 

The local Gibbs property proves that $\nu_{\gamma}$ gives positive
mass to every open set in $J$, and the mixing property shows that the 
global shift-invariant measure $\mu_{\gamma}$ gives positive mass to every open set in $\S$. We can thus copy the argument to show it is uniquely determined on each cylinder which does not intersect $\K$ (here we use the assumption that the potential satisfies \eqref{eq:Bowen} for each cylinder $J$ with empty intersection with $\K$). 

\medskip
It now remains to prove analyticity of the pressure.
Equality~\eqref{equ-zc-pgamme-nutilde} gives
$z_{c}(\gamma)\le h_{\wt\nu}(\s)-\gamma\int V\,d\wt\nu$. 
Remark~\ref{rem-nutilde-J} states that $\wt\nu(J)=0$, and uniqueness of the equilibrium state shows that $\wt\nu$ cannot be this equilibrium state (otherwise we would have $\wt\nu(J)>0$). Hence, $z_{c}(\gamma)$ is strictly less than $\CP(\gamma)$. Then, we use \cite{Hennion-Herve} to get analyticity in each variable $z$ and $\gamma$, and the analytic version of the implicit function theorem (see \cite{Range}) shows that $\CP(\gamma)$ is analytic.

The proof of part 3 is easier. The divergence of $\CL_{z,\gamma}(\BBone_{J})(\xi)$ for some $\xi$ and $z=z_{c}(\gamma)$ ensures the divergence for every $\xi$, and then Lemma~\ref{lem-decreaz-zlambda} and the local Bowen condition show that $\lambda_{z,\gamma}$ goes to $\8$ as $z$ goes to $z_{c}(\gamma)$. This means that there exists a unique $Z>z_{c}(\gamma)$ such that $\lambda_{Z,\gamma}=1$. Using the work done in the proof of point 2, we let the reader check that necessarily $Z=\CP(\gamma)$ and the local equilibrium state produces a global equilibrium state (see also \cite{leplaideur1}).  

This finishes the proof of the proposition.
\end{proof}

Actually, Proposition~\ref{prop-equil-presspos} says a little bit
more. If the second assumption is satisfied, namely
$\CP(\gamma)>-\int V\,d\mu_{\K}$, then the unique equilibrium
state for $V$ in $\S$ is the measure 
obtained (using Equation \eqref{eq-meas-open-out}) 
from the unique equilibrium state $\nu_{\CP(\gamma),\gamma}$ for the
dynamical 
system $(J,T)$ and associated to the potential $\Phi_{\CP(\gamma),\gamma}$. 
Therefore, two special curves $z$ as function of $\gamma$ appear, see Figure~\ref{fig-geen-red-curve}. The first is $z_{c}(\gamma)$, and the second is $\CP(\gamma)$, defined by the implicit equality 
$$\log\lambda_{\CP(\gamma),\gamma}=0.$$ 
We claim that these curves are convex. 
\begin{figure}[htbp]
\begin{center}
\includegraphics[scale=0.5]{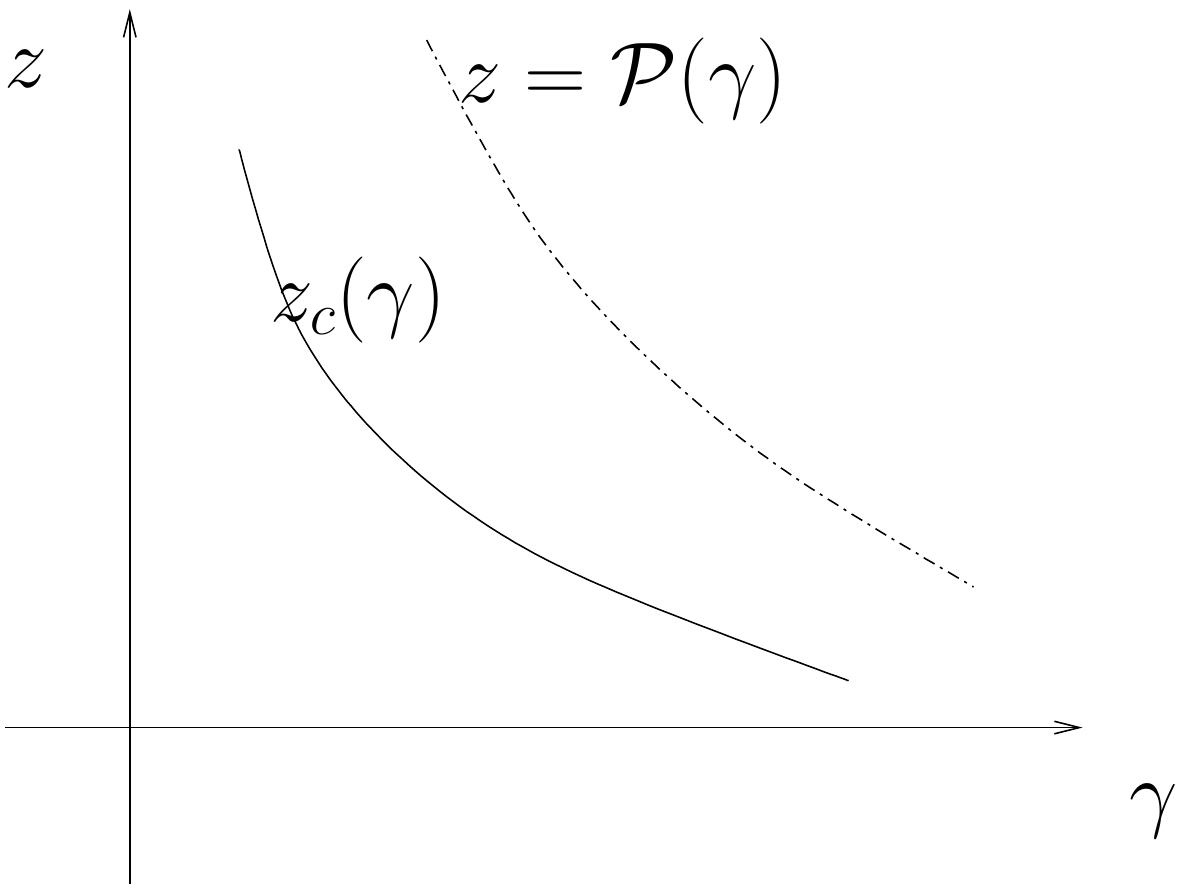}
\caption{Two important values of $z$ as function of $\gamma$.}
\label{fig-geen-red-curve}
\end{center}
\end{figure}

\subsection{Counting excursions close to $\K$}\label{subsec:counting}

Let $x\in \S$ and $n\in\N$ be such that for $k\in\llb0,n-1\rrb$, $d(\s^{k}(x),\K)\le 2^{-5}\delta_{J}$. We divide the piece of orbit  $x, \s(x), \dots, \s^{n-1}(x)$
  into pieces between accidents.
We take $b_0 = 0$ by default, and let $y^0 \in \K$ be the point 
closest to $x$. Inductively, set
$$
\begin{array}{rcl}
b_1 &=& \min\{j \ge 1 : d(\s^j(x),\K) \le d(\s^j(x),\s^j(y^0))\}, \\
&& \qquad y^1 \in \K \text{ is point closest to } \s^{b_1}(x). \\[2mm]
b_2 &=& \min\{j \ge 1 : d(\s^{j+b_1}(x),\K) \le d(\s^{j+b_1}(x),\s^{j}(y^1))\},  \\
&& \qquad y^2 \in \K \text{ is point closest to } \s^{b_1+b_2}(x). \\[2mm]
b_3 &=& \min\{j \ge 1 : d(\s^{j+b_1+b_2}(x),\K) \le d(\s^{j+b_1+b_2}(x),\s^j(y^2) \}, \\
&& \qquad y^3 \in \K \text{ is point closest to } \s^{b_1+b_2+b_3}(x) \\[2mm]
\vdots &  & \qquad\qquad \vdots\qquad\qquad \qquad\qquad \vdots
\end{array}
$$
and $d_j = -\log_2 d(\s^{\sum_{i < j} b_i}(x),\K) =
 -\log_2 d(\s^{\sum_{i < j} b_i}(x),y^{j-1})$ 
expresses how close the image of $x$  is to $\K$ at the $j-1$-st accident.

Following Proposition~\ref{prop-time-accident}, $d_{j}-b_{j}$ is of the form $3^{\eps_{j}}2^{k_{j}}$, with $\eps_{j} \in \{0,1\}$ and $d_{j+1}>d_{j}-b_{j}$ by definition of an accident.
One problem we will have to deal with, is to count the possible accidents during a very long piece of orbit: if we know $d_{j}-b_{j}$ can we determine
the possible values of $d_{j}$? 
As it is stated in Subsection~\ref{subsec-genresult-TM}, accidents occur at bispecial words which have to be prefixes of $\tau_{n}\bar\tau_{n}\tau_{n}$ or $\bar\tau_{n}\tau_{n}\bar\tau_{n}$, and are words of the form $\tau_{k}$, $\bar\tau_{k}$, $\tau_{k}\bar\tau_{k}\tau_{k}$ or $\bar\tau_{k}\tau_{k}\bar\tau_{k}$.

\medskip
From now on, we pick some non-negative potential $V$ and assume it satisfies hypotheses of Proposition~\ref{prop-equil-presspos}. 
Namely, our potentials are of the form $V(x)=\frac1{n^{a}}+o(\frac1{n^{a}})$ if $\log_{2}(d(x,\K))=-n$. 
They satisfy hypotheses of Proposition~\ref{prop-equil-presspos}, and furthermore, the Birkhoff sums are locally constant.

Moreover, for $x$ and $y$ in $J$ coinciding until $n=\tau(x)=\tau(y)$, the assumption $d(\K,J)=\delta_{J}=d(x,\K)=d(y,\K)$ shows that for every $j\le n$, 
$$d(\s^{j}(x),\K)=d(\s^{j}(y),\K)$$
holds. Hence $\Phi_{.,\gamma}$ satisfies the local Bowen property 
\eqref{eq:Bowen}.

Let $x$ be a point in $J$. We want to estimate $(\CL_{z,\gamma}\BBone_{J})(x)$.
Let $y$ be a preimage of $x$ for $T$. 
To estimate $\Phi(y)$, we decompose the orbit $y, \s(y), \dots, \s^{\tau(y)-1}(y)$ where $\s^j(y)$ is reasonably far away from $\K$
(let $c_r \ge 0$ be the length of such piece) and {\em excursions} close 
to $\K$.

\begin{definition}
\label{def-excursion}
An excursions begins at $\xi := \s^s(y)$ when $\xi$ starts as 
$\rho_0$ or $\rho_1$ for at least $5-\log_{2}\delta_{J}$ digits
(\ie $d(\xi', \rho_0)$ and $d(\xi', \rho_0) \le\delta_J 2^{-5}$)
and ends at $\xi' := \s^t(y)$ where $t > s$ is the minimal integer such that
$d(\xi', \K) > \delta_J 2^{-5}$.
\end{definition}

If $\s^i(y)$ is very close to $\K$, then due to minimality of $(\K, \s)$
it takes a uniformly bounded (from above) number of iterates for an excursion to begin.

Note that each cylinder for the return map $T$ is characterized by
a {\em path}
$$
c_0, \underbrace{b_{1,1}, b_{1,2}, \dots, b_{1,N_1}}_{\text{\tiny first excursion}}, c_1, 
\underbrace{b_{2,1}, \dots, b_{2,N_1}}_{\text{\tiny second excursion}}, c_2, 
\dots , c_{M-1}, 
\underbrace{b_{M,1}, b_{M,2}, \dots, 
b_{M,N_1}}_{\text{\tiny $M$-th excursion}},
c_M.
$$  
Any piece of orbit between two excursions or before the first excursion or after the last excursion is called a \emph{free path}.
Let $s_r$ and $t_r$ be the times where the $r$-th free path and 
 $r$-th excursion begin. Since $J$ is disjoint from $\K$, each piece
$c_r$ of free path takes at least two iterates, so $c_r \ge 2$
for $0 \le r \le M$.

\medskip
Due to the locally constant potential we are considering, $(\CL_{z,\gamma}\BBone_{J})(x)$  is independent of the point $x$ where it is evaluated. Hence, for the rest of the proofs in this section, unless it is necessary, we shall just write $\CL_{z,\gamma}\BBone_{J}$. 
Our strategy is to glue on together paths in functions of their free-paths and the numbers of accidents during an excursion. This form clusters and the contribution of a cluster considering $N$ accidents is of the form 

\begin{equation}\label{eq:excur_sum}
E_{z,\gamma}(\BBone_{J}) :=
\sum_{N \ge 1} 
\underbrace{\sum_{\stackrel{\text{allowed}}{(b_j)_{j=1}^N,\ (d_j)_{j=1}^N}}
\exp\left(-\gamma \sum_{j=1}^N S_jV \right) \exp\left(-\sum_{j=1}^N b_jz \right)  D_{N}}_{A_N} ,
\end{equation}
where $S_jV$ is the Birkhoff sum of the potential after the $j^{th}$ accident but before the next one and the quantity $D_{N}=\disp e^{\varphi_{N}-(d_{N}-b_{N})z}$ is the contribution of the last part of the orbit after the $N^{th}$ accident. By definition of an accident, this contribution is larger than if there would be no accident.
Therefore  for non-negative $z$, $\frac{e^{-(d_{N}-b_{N})z}}{(d_{N}-b_{N})^{\gamma}} \le D_N \le 1$.
The quantity $A_{N}$ is the sum of the contribution of the cluster with $N$ accident.

Thus we have
\begin{eqnarray}\label{equ2-clzgam-Ezgam}
(\CL_{0,\gamma}\BBone_{J})(x)  &= &  
\left(\sum_{c_0 \ge 5} \sum_{\stackrel{\text{\tiny free}}{c_0-\text{\tiny paths}}}
e^{-\gamma \sum_{n=0}^{c_0-1} V(\s^n(y)) - c_0z} \right) \times \nonumber \\[3mm]
&& \hspace{-1cm} \left(\sum_{M \ge 0} \left[\left(\sum_{ (c_r)_{r=1}^M }
\sum_{\stackrel{\text{\tiny free}}{c_r-\text{\tiny paths}}}
e^{-\gamma \sum_{n=0}^{c_r-1} V(\s^{n+s_r}(y)) - c_rz}\right) \times \left(E_{z,\gamma}(\BBone_{J})\right)^M\right]\right). 
\end{eqnarray}

\subsection{The potential \boldmath $n^{-a}$: \unboldmath Proofs of Theorems~\ref{theo-thermo-a>1} and \ref{theo-thermo-a<1}
\label{subsec-1/na} }

\subsubsection{Proof of Theorem~\ref{theo-thermo-a>1}}
Here we deal with the case $a > 1$ and $V(x) = n^{-a}$ if $d(x, \K) = 2^{-n}$.

\begin{proof}[Proof of Theorem~\ref{theo-thermo-a>1}]
Since $a > 1$, 
$$
\sum_{n=d-b+1}^{d}\frac1{n^{a}}\asymp \int_{d-b}^{d}\frac1{x^{a}}dx
=\frac1{a-1} \left( \frac{1}{(d-b)^{a-1}} - \frac{1}{d^{a-1}} \right)
\le \frac{1}{a-1} < \8,
$$ 
for all values of $b < d$.
To find a lower bound for $E_{z,\gamma}(\BBone_{J})$
in \eqref{eq:excur_sum}, it suffices to take only excursions with a single
accident, and sum over all possible $d_1$ with $d_1-b_1 = 2^k$.
Then 
$$
E_{z,\gamma}(\BBone_{J}) \ge
\sum_k e^{-\gamma/(a-1)} = \8,
$$
regardless of the value of $\gamma> 0$.

By Proposition~\ref{prop-equil-presspos} (part 1) we get $\CP(\gamma) > z_c(\gamma) \ge 0=-\gamma\int V\, d\mu_{\K}$. 
Then Proposition~\ref{prop-equil-presspos} (part 3) 
ensures that there is no phase transition and that $\gamma\mapsto \CP(\gamma)$ is positive and analytic on $[0,\8)$. 

To finish the proof of Theorem~\ref{theo-thermo-a>1},
 we need to compute $\lim_{\gamma\to\8}\CP(\gamma)$. Let $\mu_{\gamma}$ be the unique equilibrium state for $-\gamma V$. 
Then 
$$
\frac{\CP(\gamma)}{\gamma}=\frac{h_{\mu_{\gamma}}}\gamma-\int V\,d\mu_{\gamma},
$$
which yields $\limsup_{\gamma\to\8}\frac{\CP(\gamma)}{\gamma}\le 0$, hence $\limsup_{\gamma\to\8}\CP(\gamma)\le 0$. 
On the other hand, $\CP(\gamma)\ge 0=h_{\mu_{\K}}-\int V\,d\mu_{\K}$, hence $\liminf_{\gamma\to\8}\CP(\gamma)\ge 0$. 
\end{proof}

\subsubsection{Proof of Theorem~\ref{theo-thermo-a<1} for a special case.}
 Now 
take $a \in (0,1)$ and
$$
V(x) = n^{-a} \text{ if } d(x,\K) = 2^{-n},
$$
so
$$
\Phi_{z,\gamma}(x) = -\gamma S_n V(x) - n z = -\gamma \sum_{k=1}^n k^{-a} - nz.
$$
The potential is locally constant on sufficiently small
cylinder sets. 
It thus satisfies the local Bowen condition \eqref{eq:Bowen} and 
the hypotheses of Proposition~\ref{prop-equil-presspos} hold.  

Recall that 
$$
\sum_{n=d-b+1}^{d}\frac1{n^{a}}\asymp \int_{d-b}^{d}\frac1{x^{a}}dx=\frac1{1-a}\left(d^{1-a}-(d-b)^{1-a}\right),
$$ 
and we shall replace the discrete sum by the integral. 
The error involved in this can be incorporated in the changed 
coefficient $(1\pm\eps)\gamma$. 

Our goal is to prove that $z_{c}(\gamma)=0$ (for every $\gamma$) and that $\CL_{0,\gamma}(\BBone_{J})(x) \to 0$ as $\gamma \to \8$ (for any $x \in J$). This will prove that there is a phase transition.

\begin{lemma}
\label{lem-diveznegna}
The series $(\CL_{z,\gamma}\BBone_{J})(\xi)$ diverges for $z<0$.
\end{lemma}

\begin{proof}
We employ notations from \eqref{eq:excur_sum} with our new $V$. In the
full shift all orbits appear, and we are counting here only orbits which have only one excursion close to $\K$ without accident. For each $j$, we consider a piece of orbit of length $2^{k+1}(1+2j)$, coinciding with a piece of orbit 
within $\K$, and then ``going back'' to $J$. 
The quantity $E_{z,\gamma}(\BBone_{J})$ is larger than the contribution of these excursions, which is 
$$
A_{1}^{k}(z)\ge \sum_{j=1}^{\8} e^{-\frac\gamma{1-a}((2^{k+1}(1+2j))^{1-a}-1)-2^{k+1}(1+2j)z}.
$$
As $a<1$, $-2jz$ is eventually  larger than $(2^{k+1}(1+2j))^{1-a}$ for $z<0$ and the series trivially diverges. Then, $E_{z,\gamma}(\BBone_{J})$ diverges as well, and \eqref{equ2-clzgam-Ezgam} shows that $\CL_{z,\gamma}(\BBone_{J})$ diverges for every initial point $x \in J$. 
\end{proof}

Let us now consider the case $z=0$. 
As we are now looking for upper bounds, we can consider the $b_{j}$'s and the $d_{j}$'s as independent and sum over all possibilities (and thus forget the condition $d_{j+1}>d_{j}-b_{j}$).
Note that we trivially have $D_{N}\le 1$. 

For a piece of orbit of length $d$ and with an accident at $b$, $d-b=2^{k}$, the possible values of $d$'s are among $2^{k}(1+\frac{j}2)$, $j\ge 1$, and then $b=2^{k-1}j$. If $d-b=3 \cdot 2^{k}$, then the possible values of $d$'s are among  $2^{k}(1+\frac{j}2)$ with $j\ge 5$ and then $b=2^{k}(\frac{j}2-2)$. 

Let 
$$
B(z) :=
\sum_{k=4}^{\8} \sum_{j = 1}^{\8} e^{-\frac\gamma{1-a}\left(\left(2^{k}(1+\frac{j}2)\right)^{1-a}-2^{k(1-a)}\right)-j2^{k-1}z}.
$$
and 
$$
C(z):=\sum_{k=4}^{\8} \sum_{j = 5}^{\8} e^{-\frac\gamma{1-a}\left(\left(2^{k}(1+\frac{j}2)\right)^{1-a}-3^{1-a}2^{k(1-a)}\right)-2^{k-1}(j-4)z}.
$$
The quantity $B(z)$ is an upper bound for the cluster with one excursion of the form $d-b=2^{k}$, and $C(z)$ is an upper bound for the cluster with one excursion of the form $d-b=3\cdot 2^{k}$. 

Then multiplying $N$ copies 
to estimate from above the contribution of excursion with $N$ accidents
we get
$$
E_{z,\gamma}(\BBone_{J})\le \sum_N (B(z)+C(z))^N.
$$
Hence
\begin{align*}\label{eq:Lsum}
&(\CL_{0,\gamma}\BBone_{J})(x) \ \le \  
\left(\sum_{c_0 \ge 5} \sum_{\stackrel{\text{\tiny free}}{c_0-\text{\tiny paths}}}
e^{-\gamma \sum_{n=0}^{c_0-1} V(\s^n(y)) - c_0z} \right) \times \nonumber \\[3mm]
& \quad \left(\sum_{M \ge 0} \left[\left(\sum_{ (c_r)_{r=1}^M }
\sum_{\stackrel{\text{\tiny free}}{c_r-\text{\tiny paths}}}
e^{-\gamma \sum_{n=0}^{c_r-1} V(\s^{n+s_r}(y)) - c_rz}\right) \times \left(\sum_{N\ge1}(B(z)+C(z))^N\right)^M\right]\right), 
\end{align*}
where the sum over $(c_r)_{r=1}^M$  is $1$ by convention if $M = 0$.
The first factor (the sum over $c_0$) indicates the first cluster of 
free paths, and $c_0 \ge 5$ by our choice of the distance $\delta_J$.

Note that for the free pieces between excursions the orbit is relatively far from
$\K$, so there is $\eps > 0$ depending only on $\delta_J$ such that
\begin{equation}
\label{equ-encadre-freepath }
-c_r (\gamma + z) \le \sum_{n=0}^{c_r-1} -\gamma V(\s^{n+s_r}(y)) - c_rz \le -c_r (\eps \gamma + z).
\end{equation}

An upper bound for $\CL_{z,\gamma}(\BBone_{J})$ is obtained by taking an upper 
bound for $B$  and $C$ and majorizing the sum over the $c_{r}$ free paths by taking the sum over all the $c$ and the upper bound in \eqref{equ-encadre-freepath }. 

\begin{proof}[Proof of Theorem~\ref{theo-thermo-a<1}]
Lemma~\ref{lem-diveznegna} shows that for every $\gamma$, $z_{c}(\gamma)\ge 0$. Our goal is to prove that $B(0)+C(0)$ can be made as small as wanted by choosing $\gamma$ sufficiently large. This will imply that $z_{c}(\gamma)=0$ for sufficiently large $\gamma$ and that the unique equilibrium state is $\mu_{\K}$. 
We compute $B(0)$ leaving the very similar computation
for  $C(0)$ to the reader. 

Apply the inequality $1+u \ge 1+\log(1+u)$ for the value of $u$ such that
$1+u = (1+\frac{j}{2})^{1-a}$, to obtain 
$(1+\frac{j}{2})^{1-a} - 1 \ge \log(1+\frac{j}{2})^{1-a}$, whence
$e^{(1+\frac{j}{2})^{1-a} - 1} \ge (1+\frac{j}{2})^{1-a}$.
Raising this to the power $-\frac{\gamma}{1-a} 2^{k(1-a)}$ and summing over $j$, we get
\begin{eqnarray*}
\sum_{j=1}^\8 e^{-\frac{\gamma}{1-a} 2^{k(1-a)}((1+\frac{j}{2})^{1-a}-1)}
&\le& \sum_{j=1}^\8(1 +\frac{j}{2})^{-\gamma 2^{k(1-a)}} \\
&\le& \left(\frac23\right)^{\gamma 2^{k(1-a)}} +
 \int_1^\8\frac{dx}{ (1 +\frac{x}{2})^{-\gamma 2^{k(1-a)}} }\\
 &=& \left(1 + \frac{3}{\gamma 2^{k(1-a)}-1} \right) \left(\frac23\right)^{\gamma 2^{k(1-a)}}.
\end{eqnarray*}
Therefore
$$
B(0) = \sum_{k=4}^{\8}\sum_{j=1}^{\8}e^{\frac\gamma{1-a}2^{k(1-a)}(1-(1+\frac{j}2)^{1-a})}\le
\sum_{k=4}^{\8}  \left(1 + \frac{3}{\gamma 2^{k(1-a)}-1} \right) \left(\frac23 \right)^{\gamma 2^{k(1-a)}}
$$
is clearly finite and tends to zero as $\gamma \to \8$.

Now to estimate $(\CL_{z,\gamma} \BBone_{J})(x)$, we have to sum over the 
free periods as well and we have
\begin{eqnarray}\label{eq:bound}
(\CL_{0,\gamma} \BBone_{J})(x) &\le &
\left( \sum_{c \ge5} 2^c e^{-\eps \gamma c}  \right) \cdot  
 \sum_{M \ge 0} \left( \sum_{c \ge 1} 
2^c e^{-c \eps \gamma } (E_{0,\gamma}\BBone_{J})(\xi)  \right)^M  \nonumber \\
&\le&
\frac{32 e^{-5\eps\gamma}}{1-2e^{-\eps\gamma}} \sum_{M \ge 0} \left( \sum_{c \ge 1} 
2^ce^{-c \eps \gamma } \sum_{N \ge 1} (B(0)+C(0))^N  \right)^M \!\!\! .
\end{eqnarray}
The term in the brackets still tends to zero as $\gamma \to \infty$,
and hence is less than $1$ 
for $\gamma > \gamma_0$ and some  sufficiently large $\gamma_0$.
The double sum converges for such $\gamma$, 
so the critical $z_c(\gamma) \le 0$ for $\gamma \ge\gamma_0$.

Lemma~\ref{lem-diveznegna} shows that $z_{c}(\gamma)$ is always non-negative. 
Therefore $z_{c}(\gamma)=0$ for every $\gamma\ge \gamma_{0}$. 
In fact, for $\gamma$ sufficiently large (and hence $e^{-5\eps\gamma}$ is sufficiently small), 
the bound \eqref{eq:bound} is less than one:  for every $x \in J$, 
$(\CL_{0,\gamma}\BBone_{J})(x) < 1$. This means that
$\log\lambda_{0,\gamma}$,
\ie the logarithm of the 
spectral radius of $\CL_{z,\gamma}$, becomes zero at some value of $\gamma$, say $\gamma_2$. 

Lemma~\ref{lem-decreaz-zlambda} says that the spectral radius decreases in $z$. On the other hand the pressure $\CP(\gamma)$ is non-negative because $\int V\,d\mu_{\K}=0$. Moreover, the curve $z=\CP(\gamma)$ is given by the implicit equality $\lambda_{\CP(\gamma),\gamma}=1$. Therefore, the curve 
$\gamma \mapsto \CP(\gamma)$ is below the curve $\gamma \mapsto \log\lambda_{0,\gamma}$. 
Thus it must intersect the horizontal axis at some $\gamma_1\in[\gamma_{0},\gamma_{2}]$ (see Figure~\ref{fig-gamma12}).

\begin{figure}[htbp]
\begin{center}
\includegraphics[scale=0.5]{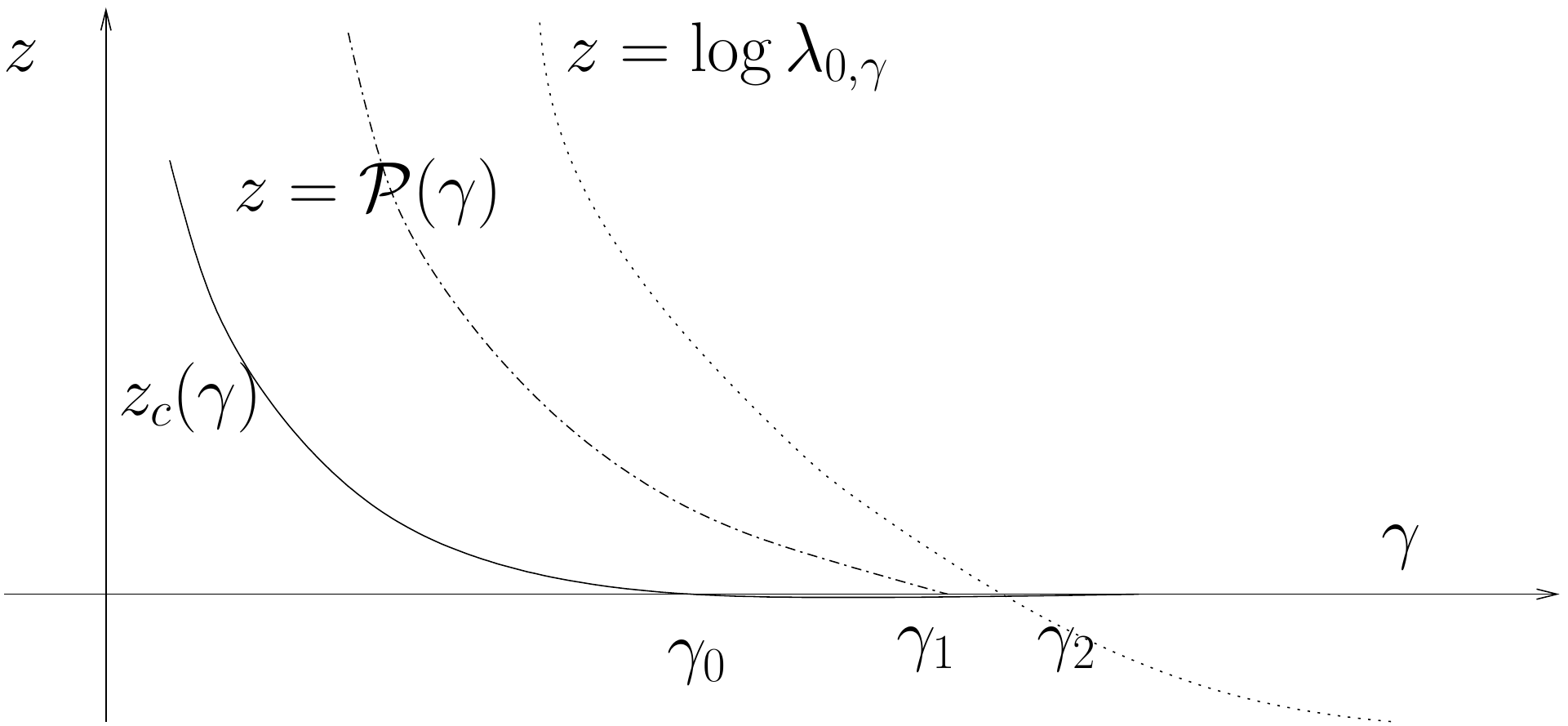}
\caption{Phase transition at $\gamma_1$}
\label{fig-gamma12}
\end{center}
\end{figure}

For $\gamma>\gamma_{1}$ convexity yields $\CP(\gamma)=0$, hence the
function is not analytic at $\gamma_1$ and we have an ultimate 
phase transition (for $\gamma=\gamma_1$). 
Analyticity for $\gamma<\gamma_{1}$ follows from Proposition~\ref{prop-equil-presspos}.
\end{proof}

\subsubsection{Proof of Theorem~\ref{theo-thermo-a<1} for the general case}

Now we consider $V$ such that
$$
V(x) = n^{-a}+o(n^{-a}) \quad \text{ if } d(x,\K) = 2^{-n}.
$$
For every fixed $\eps_{0}$, there exists some $N_{0}$ such that for 
every $n\ge N_{0}$ and for $x$ such that $d(x,\K)=2^{-n}$,
$$
\left|V(x)-\frac1{n^{a}}\right|\le \frac{\eps_{0}}{n^{a}}.
$$
We can incorporate this perturbation in the free path, 
assuming that any path with length less than 
$N_{0}$ is a free path. Then all the above computations are valid provided we replace 
$\gamma$ by $\gamma(1\pm\eps_{0})$. This does not affect the results.

\subsection{The proof of Theorem~\ref{theo-super-MP}}

As a direct application of Theorem~\ref{theo-thermo-a<1}, we can give a 
version of the Manneville-Pomeau map with a neutral Cantor set instead of a neutral fixed point.

\begin{proof}[Proof of Theorem~\ref{theo-super-MP}]
Pick $a > 0$, 
and consider $V$ and $\gamma_{1}$ as in Subsection~\ref{subsec-1/na} (only for $a<1$).  For $a>1$ we pick any positive $\gamma_{1}$. Define the canonical projection $\Pi:\S\rightarrow [0,1]$ 
by the dyadic expansion:
$$
\Pi(x_{0},x_{1},x_{2},\ldots)=\sum_{j}\frac{x_{j}}{2^{j+1}}.
$$
It maps $\K$ to a Cantor subset of $[0,1]$.
Only dyadic points in $[0,1]$ have two preimages under $\Pi$, namely 
$x_{1}\ldots x_{n}10^{\8}$ and $x_{1}\ldots x_{n}01^{\8}$ have the same image. 

\begin{lemma}
\label{lem-pot-W1na}
There exists a potential $W:\S\to\R$ such that 
$$W(x)=\frac1{n^{a}}+o(\frac1{n^{a}}) \quad \text{ if } d(x,\K)=2^{-n},$$
and it is continuous at dyadic points:
$$W(x_{1}\ldots x_{n}10^{\8})=W(x_{1}\ldots x_{n}01^{\8}),$$
and is positive everywhere except on $\K$ where it is zero.
\end{lemma}
\begin{proof}
Let us consider the multi-valued function $V\circ\Pi^{-1}$ on the interval. It is uniquely defined at 
each non-dyadic point. For a dyadic point, consider the two preimages 
$x_{1}\ldots x_{n}10^{\8}$ and $x_{1}\ldots x_{n}01^{\8}$  in $\S$.

\paragraph{\it Case 1.} The word $x_{1}\ldots x_{n}$ (which is $\K$-admissible) has a single suffix in $\K$, say $0$. This means that $x_{1}\ldots x_{n}0$ is an admissible word for $\K$ but not $x_{1}\ldots x_{n}1$. 
Let $\ul{x}^{-}:=x_{1}\ldots x_{n}01^{\8}$ and  $\ul{x}^{+}:=x_{1}\ldots x_{n}10^{\8}$. Then
 \begin{equation}
\label{equ1-modif1pot}
d(\ul{x}^{+},\K)=2^{-n}>d(\ul{x}^{-},\K)>2^{-n-4},
\end{equation}
 where the last inequality comes from the fact that $x_{1}\ldots x_{n}0111$ is not admissible for $\K$.  

We modify the potential $V$ on the right side hand of the dyadic point $\Pi(x_{1}\ldots x_{n}10^{\8})$ as indicated on Figure~\ref{fig-modif1}.

\begin{figure}[htbp]
\begin{center}
\includegraphics[scale=0.5]{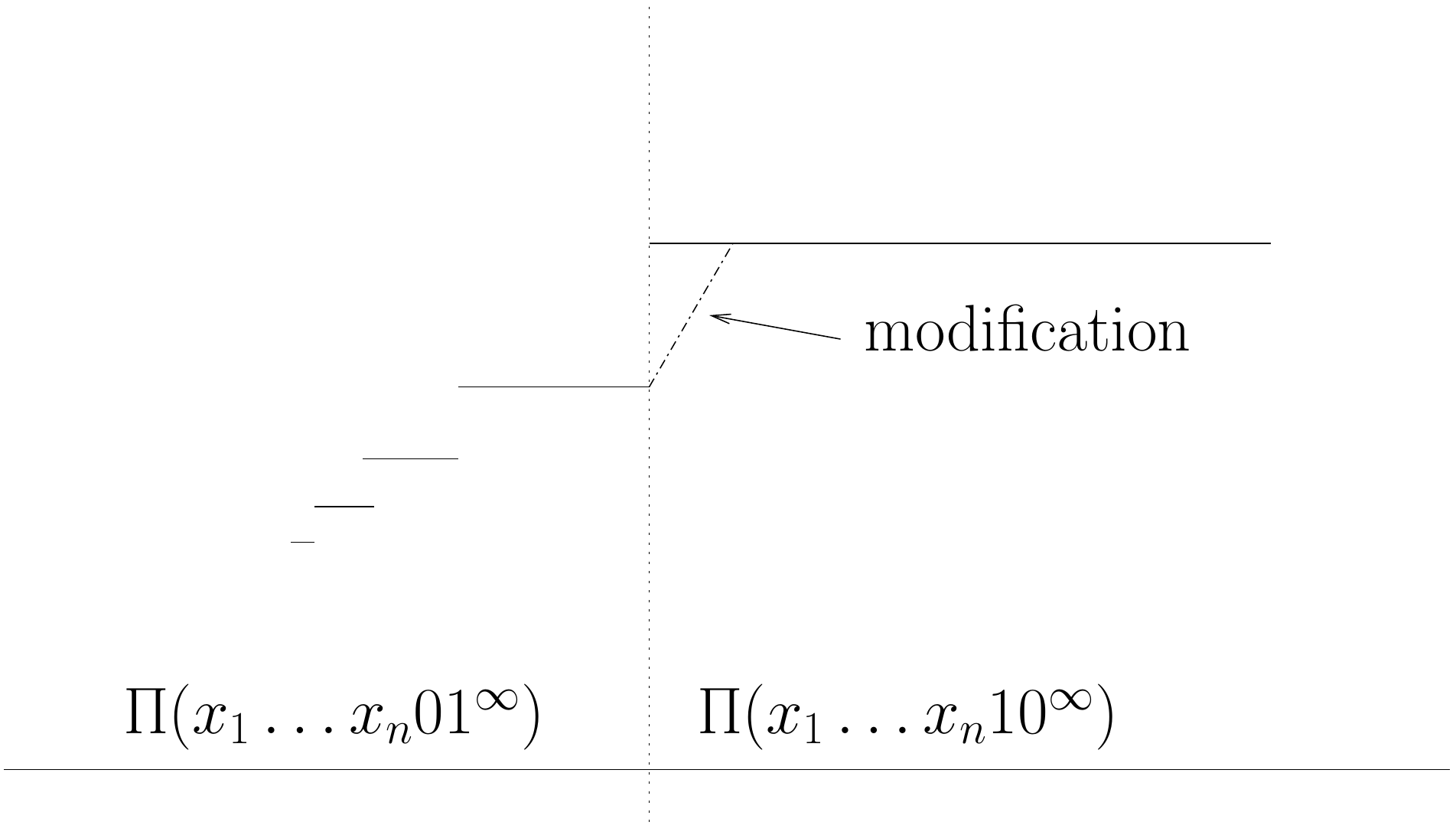}
\caption{The modification for words with a single suffix}
\label{fig-modif1}
\end{center}
\end{figure}

The inequalities of \eqref{equ1-modif1pot}  yield 

$$
V(\ul{x}^{-})= \frac1{(n+k)^{a}}=\frac1{n^{a}}-\frac{ak}{n^{a+1}}+o(\frac1{n^{a+2}})
=V(\ul{x}^{+})+o(V(\ul{x}^{+})),
$$
where $k$ is an integer in $[1,4]$. 
As the modification is done ``convexly'', the new potential $W$ satisfies for these modified points
$$W(x)=\frac1{n^{a}}+o(\frac1{n^{a}})\quad \text{ if }d(x,\K)=2^{-n}.$$

\paragraph{\it Case 2.} The word $x_{1}\ldots x_{n}$ (which is $\K$-admissible) has two suffixes in $\K$. It may be that $\ul{x}_{+}$ and $\ul{x}_{-}$ are at the same distance to $\K$ (see Figure~\ref{fig-modif2}). Then we do not need to change the potential around this dyadic point. 

\begin{figure}[htbp]
\begin{center}
\includegraphics[scale=0.5]{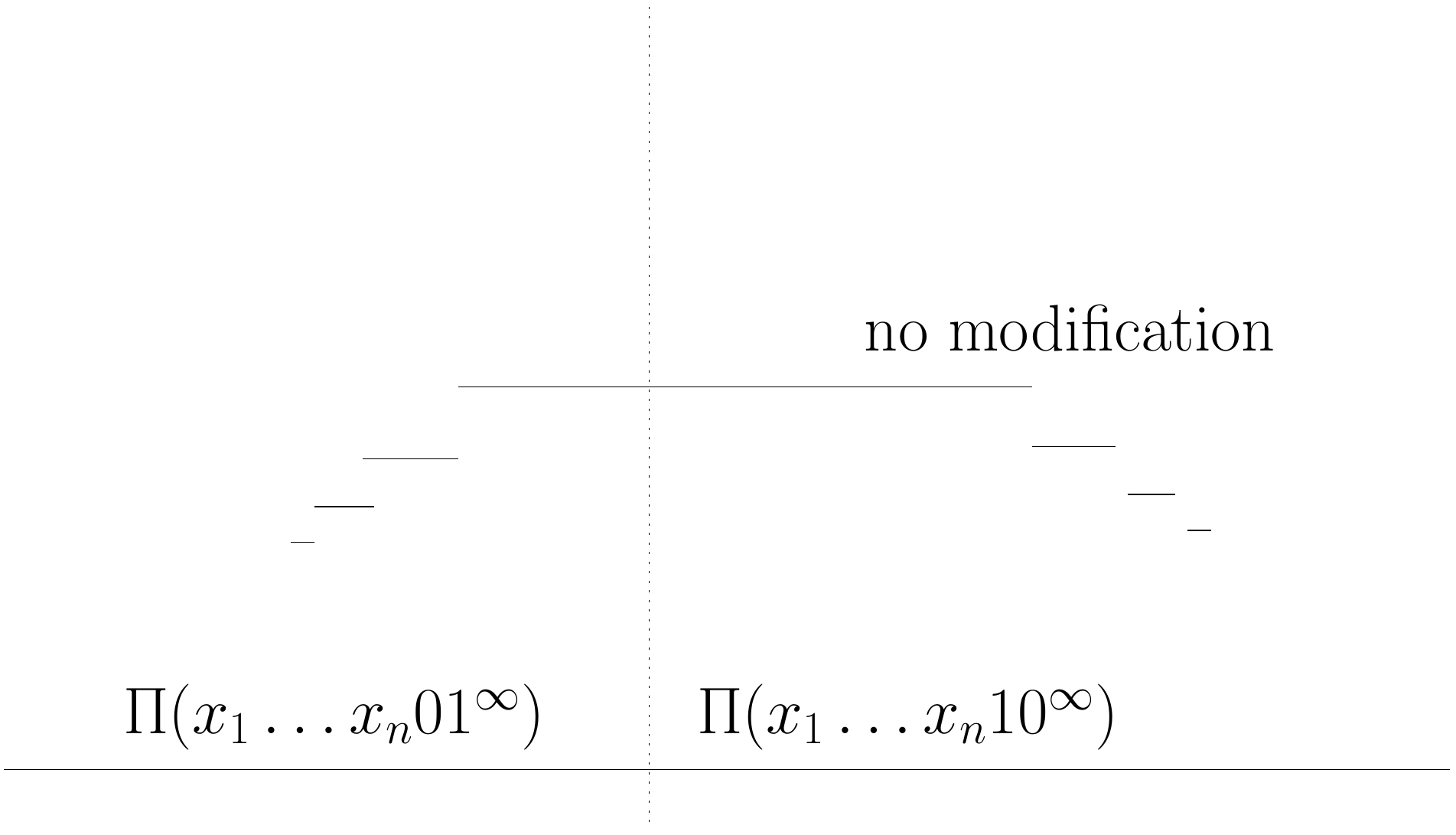}
\caption{No modification with two different suffixes}
\label{fig-modif2}
\end{center}
\end{figure}

If $V(\ul{x}^{+})\neq V(\ul{x}^{-})$, neither $x_{1}\ldots x_{n}0111$ nor 
$x_{1}\ldots x_{n}1000$ are admissible for $\K$ and we modify the potential linearly in that 
region in the interval as Figure~\ref{fig-modif3}. 

\begin{figure}[htbp]
\begin{center}
\includegraphics[scale=0.5]{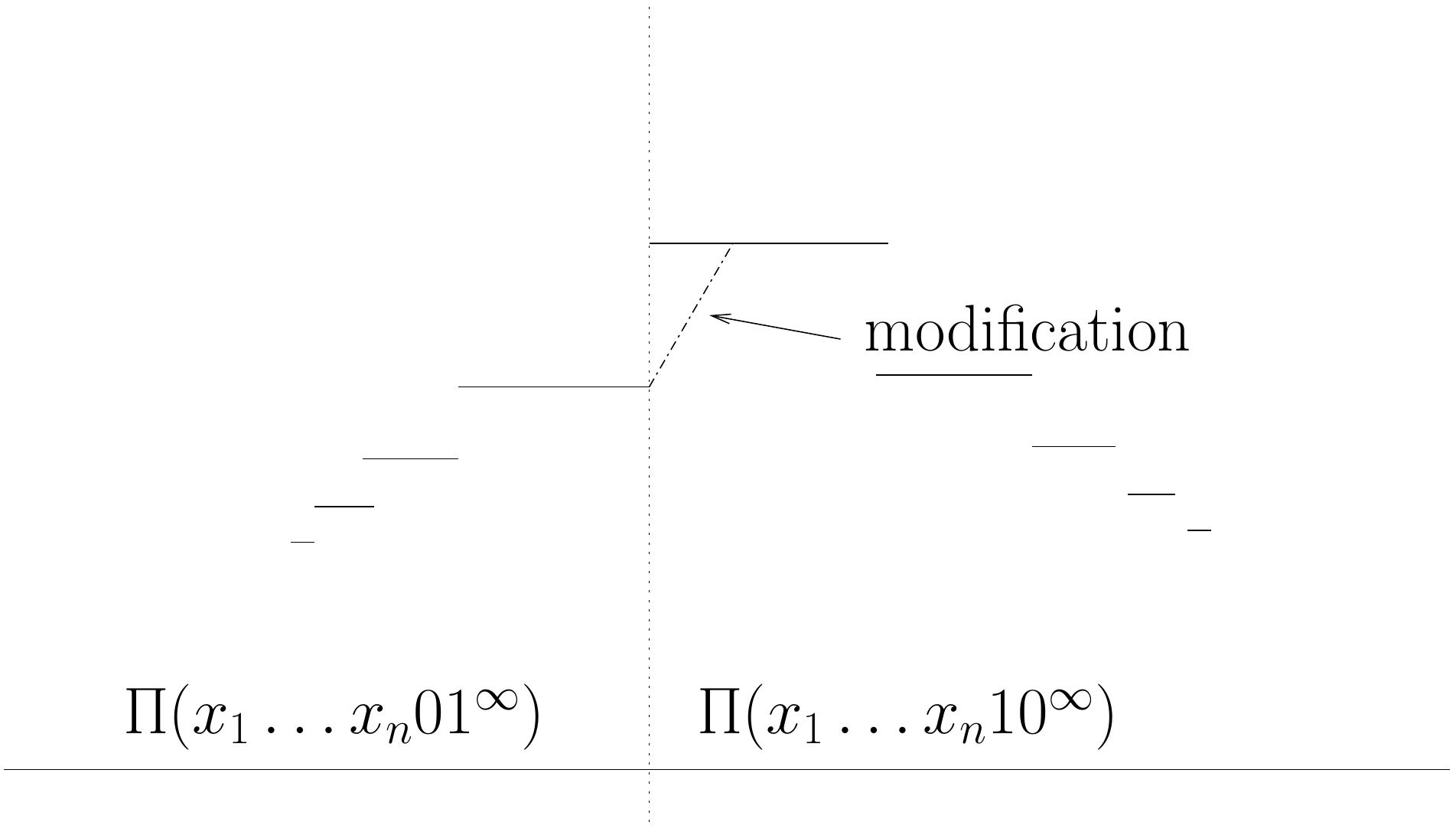}
\caption{modification with two different suffixes}
\label{fig-modif3}
\end{center}
\end{figure}

Again we have
$$
V(\ul{x}^{+})=\frac1{(n+j)^{a}}=\frac1{n^{a}}+o(\frac1{n^{a}})
\quad \text{ and } \quad 
V(\ul{x}^{-}) = \frac1{(n+k)^{a}} = \frac1{n^{a}}+o(\frac1{n^{a}}),
$$ 
where $j$ and $k$ are different integers in $\{ 1,2,3,4\}$. 
Hence, for these points too, $W$ satisfies
$$
W(x)=\frac1{n^{a}}+o(\frac1{n^{a}}) \quad \text{ if }d(x,\K)=2^{-n}.
$$
Positivity of $W$ away from $\K$ follows from the positivity of $V$ and the way of modifying it to get $W$. Clearly $W$ vanishes on $\K$.
\end{proof}

  \paragraph{\it The case $a<1$}
Continuing the proof of Theorem~\ref{theo-super-MP},
the eigen-measure $\nu_{a}$ in $\S$ is a fixed point for the adjoint 
of the transfer operator for $\gamma_{1}$ (the pressure vanishes at $\gamma_{1}$) for the potential $W$. 
As the potential $W$ is continuous and the shift is Markov, such a measure always exists. It is conformal in the sense that  
\begin{equation}
\label{equ1-meas-conformal}
\nu(\s(B))=\int_{B} e^{\CP(\gamma_{1})+\gamma_{1}W}\,d\nu_{a}=\int_{B} e^{\gamma_{1}W}\,d\nu_{a},
\end{equation}
for any Borel set $B$ on which $\s$ is one-to-one. 
Since we have a phase transition at $\gamma_{1}$, 
$\CP(\gamma_1) = 0$. Note also that $W$ is positive everywhere except on $\K$ where it vanishes. 

Now consider the measure $\Pi_{*}(\nu_{a})$ and its distribution function
$$\theta_{a}(x) := 
\nu_{a}([0^\8,\Pi(x)))=\nu_{a}([0^\8,\Pi(x)]),
$$
the last equality resulting from the fact that $\nu_{a}$ is non-atomic. 
We emphasize that $\Pi$ maps the lexicographic order in $\S$ 
to the usual order on the unit interval $[0,1]$. 
This enables us to define \emph{intervals} in $\S$, for which we will 
use the same notation $[x,y]$.

Let us now compute the derivative of $f_{a}$ define by 
$$f_{a}:=\theta_{a}\circ\Pi\circ\s\circ\Pi^{-1}\circ \theta_{a}^{-1}$$
at some point $x \in [0,1]$. 
For $h$ very small we define $y$ and $y_{h}$ in $[0,1]$ such that 
$\Pi_{*}\nu_{a}([0,y])=x$ and $\Pi_{*}\nu_{a}([0,y_{h}])=x+h$. 
Also define $\ul y$ and $\ul{y_{h}}$ such that $\Pi(\ul y)=y$ and $\Pi(\ul{y_{h}})=y_{h}$. 
Then we get 
\begin{eqnarray*}
\frac{f_{a}(x+h)-f_{a}(x)}{h}&=&\frac{\nu_{a}([\s(\ul y),\s(\ul{y_{h}})])}{\nu_{a}([\ul y,\ul{y_{h}}])} \\
&=& \frac{\nu_{a}(\s([\ul y,\ul{y_{h}}]))}{\nu_{a}([\ul y,\ul{y_{h}}])}\\
&=&\frac1{\nu_{a}([\ul y,\ul{y_{h}}])}\int_{[\ul y,\ul{y_{h}}]} e^{\gamma_{1}W}\,d\nu_{a}\ \rightarrow_{h\to 0} \ e^{\gamma_{1}W(y)}.
\end{eqnarray*}
This computation is valid if $\Pi^{-1}(y)$ is uniquely determined (namely $y$ is not dyadic). If $y_{h}$ is dyadic for some $h$, then we choose for $\ul{y_{h}}$ the one closest to $\ul y$. 

If $y$ is dyadic, then the same can be done provided we change the preimage of $y$ by $\Pi$ depending on whether we compute left or right derivative. 
Nevertheless, the potential $W$ is continuous at dyadic points, hence $f_{a}$ has left and right derivative at every dyadic points and they are equal.

We finally get $f'_{a}(x)=e^{\gamma_{1}W\circ\Pi^{-1}(x)}$ (this make sense also for dyadic points) and then $\log f'_{a}(x)=\gamma_{1}W\circ\Pi^{-1}(x)$. Therefore $f_{a}$ is $\CC^1$ and as $W$ is positive away from $\K$ and zero on $\K$, $f_{a}$ is expanding away from $\wt\K:=\theta_{a}\circ\Pi(\K)$ and is indifferent on $\wt\K$. 
For $t\in[0, \infty)$, the lifted potential for $-t\log f'_{a}$ is $-t\gamma_{1}W\circ \Pi^{-1}$, 
which has an ultimate phase transition for $t=1$ and $a \in (0,1)$.
 \end{proof}

\paragraph{\it The case $a>1$}
Computations are similar to the case $a<1$, except that we have to add the 
pressure for $\gamma_{1}$. The construction is the same, but the map $f_{a}$ satisfies :
$$f'_{a}(x)=e^{\gamma_{1}W\circ\Pi^{-1}(x)+\CP(\gamma_{1})}.$$
This extra term is just a constant and then, the thermodynamic formalism for $-t\log f'_{a}$ is the same that the one for $-t\gamma_{1}W\circ\Pi^{-1}$. 

\subsection{Unbounded potentials: Proof of Theorem~\ref{theo-thermo-Vu}}
\label{subsec-unbounded2}

We know from Subsection~\ref{subsec-unbounded1} 
that $\CR$ fixes the potential $V_u$, defined in \eqref{eq:Vu}. 
In this section we set $g \equiv \alpha$, which gives $V_u = \alpha(k-1)$ on 
$(\s \circ H)^k(\S) \setminus (\s \circ H)^{k+1}(\S)$.
For the thermodynamic properties of this potential, the interesting 
case is $\alpha < 0$ (see the Introduction before the statement of 
Theorem~\ref{theo-thermo-Vu} and the Appendix).

\begin{lemma}\label{lem-vu-inte-pos}
Let $\alpha < 0$. Then 
$\int V_u d\mu \ge \int_\Sigma V_u d\mu_\K = 0$ for every shift-invariant measure probability $\mu$.
\end{lemma}

\begin{proof}
As in Lemma~\ref{lem-sigmacirch},
the set $(\s \circ H)^k(\S) = \s^{2^k-1} \circ H^k( [00] \sqcup [10] \sqcup [01] \sqcup [11])$ consists of four $2^k+1$-cylinders containing the points
$1\rho_0$, $0\rho_0$, $1\rho_1$ and $1\rho_1$ respectively, and 
they are mapped
into the two $2^k$-cylinders containing $\rho_0$ and $\rho_1$.
In other words, $(\s \circ H)^k(\S) = \s^{-1} \circ H^k(\S)$, and 
by Lemma~\ref{lem:disjoint}, its next $2^k$ shifts are pairwise disjoint.
Therefore $\mu_\K((\s \circ H)^k(\S)) = 2^{-k}$ and
 $\mu_\K((\s \circ H)^k(\S) \setminus (\s \circ H)^{k+1}(\S)) = 2^{-(k+1)}$.
Since $V_u = \alpha(k-1)$ on 
$(\s \circ H)^k(\S) \setminus (\s \circ H)^{k+1}(\S)$
this gives
$$
\int V_u \ d\mu_\K = \alpha \sum_{k \ge 0} (k-1) 2^{-(k+1)} = 
-\frac{\alpha}{2} + \alpha\sum_{k \ge 2} k 2^{-(k+1)} = 0.
$$ 
Again, since $\s^j((\s \circ H)^k(\S)$
is disjoint from $((\s \circ H)^k(\S))$
for $0 < j < 2^k$, its $\mu$-mass is at most
$2^{-k}$ for any shift-invariant probability measure $\mu$.
Since $V_u$ is decreasing in $k$ (for $\alpha < 0$), we can minimize 
the integral 
$\int V_u \ d\mu$ by putting as much mass on 
$(\s \circ H)^k(\S)$ as possible, for each $k$.
But this means that the $\mu$-mass of 
$(\s \circ H)^k(\S) \setminus (\s \circ H)^{k+1}(\S)$ becomes
$2^{-(k+1)}$ for each $k$, and hence $\mu = \mu_\K$.
\end{proof}
\begin{remark}
\label{rem-measu-discrecylin}
As a by-product of our proof, $\mu\left((\s\circ H)^{k}(\S)\right)\le 2^{-k}$
for any invariant probability $\mu$ and $k \ge 2$. $\hfill \blacksquare$
\end{remark}

For fixed $\alpha<0$, the integral $\int V_{u}\,d\mu$ is non-negative and we define  for $\gamma\ge 0$
$$
\CP(\gamma):=\sup_{\mu\ \s-inv}\left\{h_{\mu}-\gamma\int V_{u}\,d\mu\right\}.
$$

\begin{proposition}
\label{prop-exit-mesdeq-vu}
For any $\gamma\ge 0$ there exists an equilibrium state for $-\gamma V_{u}$.
\end{proposition}

To prove this proposition, we need a result on the 
accumulation value $\lim\inf_{\eps \to 0}\int V_{u}d\nu_{\eps}$ if $\{ \nu_{\eps} \}_\eps$ is a family of invariant probability
measures. 

\begin{lemma}
\label{lem-liminf}
Let $\nu_{\eps}$ be a sequence of invariant probability
measures converging to $\nu$ in the weak topology as $\eps \to 0$. 
Let us set $\nu:=(1-\beta)\mu+\beta\mu_{\K}$, where $\mu$ is an invariant probability 
measure satisfying $\mu(\K)=0$ and $\beta\in [0,1]$. 
Then,
$$\liminf_{\eps\to 0}\int V_{u}\,d\nu_{\eps}\ge (1-\beta)\int V_{u}\,d\mu.$$
\end{lemma}

\begin{proof}[Proof of Lemma~\ref{lem-liminf}]
Let us consider an $\eta$-neighborhood $O_{\eta}$ of $\K$ 
consisting of finite
union of cylinders.
Clearly $\left(\s\circ H\right)^{j} \subset O_{\eta}$
for $j = j(\eta) \ge 2$ sufficiently large (and $j(\eta) \to \infty$ as $\eta
\to 0$).

Let $\nu_{\eps}$ be an invariant probability measure. 
Following the same argument as in the proof of Lemma~\ref{lem-vu-inte-pos} and  in particular Remark~\ref{rem-measu-discrecylin}, we claim that 
$$
\int \BBone_{O_{\eta}}V_{u}\,d\nu_{\eps}\ge -\frac\alpha2\nu_{\eps}
\left(O_{\eta}\setminus (\s\circ H)(\S)\right)+\alpha\sum_{k\ge j} k2^{-(k+1)}
$$
holds.
Then we have 
\begin{eqnarray}\label{equ-mino-meas-neighK}
\int V_{u}\,d\nu_{\eps} &\ge& 
\int \BBone_{\S\setminus O_{\eta}}V_{u}\,d\nu_{\eps}-\frac\alpha2\nu_{\eps}(O_{\eta}\setminus (\s\circ H)(\S))+\alpha\sum_{k\ge j}k2^{-(k+1)} \nonumber \\
&\ge& \int \BBone_{\S\setminus O_{\eta}}V_{u}\,d\nu_{\eps}+\alpha\sum_{k\ge j}k2^{-(k+1)}.
\end{eqnarray}
Note that $\BBone_{\S\setminus O_{\eta}}V_{u}$ is a continuous function. 
Thus, $ \lim_{\eps\to 0}\int \BBone_{\S\setminus O_{\eta}}V_{u}\,d\nu_{\eps}$ exists and 
is equal to $\int \BBone_{\S\setminus O_{\eta}}V_{u}\,d\nu=(1-\beta)\int \BBone_{\S\setminus O_{\eta}}V_{u}\,d\mu$. As $\eta \to 0$, this quantity decreases and converges to $(1-\beta)\int V_{u}\,d\mu$ 
(here we use $\mu(\K)=0$).
Therefore, passing to the limit in \eqref{equ-mino-meas-neighK} first in $\eps$ and then in $\eta$ we get 
$$
\liminf_{\eps\to 0}\int V_{u}\,d\nu_{\eps}\ge (1-\beta)\int V_{u}\,d\mu.
$$
\end{proof}

\begin{proof}[Proof of Proposition~\ref{prop-exit-mesdeq-vu}]
We repeat the argument given in the proof of Proposition~\ref{prop-equil-presspos} 
and adapt it as in \cite{jacoisa}. Let $\nu_{\eps}$ be a probability measure such that 
\begin{equation}
\label{equ1-equil-potinfini}
h_{\nu_{\eps}}-\gamma\int V_{u}\,d\nu_{\eps}\ge \CP(\gamma)-\eps,
\end{equation}
and let $\nu$ be any accumulation point of $\nu_{\eps}$. 
As $V_{u}$ is discontinuous we cannot directly pass to the limit 
$\eps \to 0$ and claim that the 
integral of the limit measure is the limit of the integrals.
However, we claim that $V_{u}$ is continuous everywhere but at the four points $0\rho_{0}$, $0\rho_{1}$, $1\rho_{0}$ and $1\rho_{1}$ (see and adapt the proof of Lemma~\ref{lem-sigmacirch}). These points are in $\K$ and their orbits are dense in $\K$. We thus have to consider two cases. 
\begin{itemize}
\item $\nu(\K)=0$. Then a standard argument in measure theory says that we do not see the discontinuity, 
and passing to the limit as $\eps \to 0$ in \eqref{equ1-equil-potinfini},
$$\CP(\gamma)\ge h_{\nu}-\gamma\int V_{u}\,d\nu\ge \CP(\gamma),$$
which means that $\nu$ is an equilibrium state. 
\item $\nu(\K)>0$. In this case we can write $\nu=\beta\mu_{\K}+(1-\beta)\mu$, where $\mu$ is a $\s$-invariant probability satisfying $\mu(\K)=0$ and $\beta$ belongs to $(0,1]$. Therefore
$$h_{\nu}=\beta h_{\mu_{\K}}+(1-\beta)h_{\mu}=(1-\beta)h_{\mu}.$$
\end{itemize}
Lemma~\ref{lem-liminf} yields 
\begin{equation}
\label{equ2-equil-potinfi}
\liminf_{\eps\to 0}\int V_{u}\,d\nu_{\eps}\ge (1-\beta)\int V_{u}\,d\mu.
\end{equation} 
Hence, passing to the limit in Inequality \eqref{equ1-equil-potinfini}, Inequality \eqref{equ2-equil-potinfi} shows that 
$$\CP(\gamma)\le (1-\beta)h_{\mu}-\gamma(1-\beta)\int V_{u}\,d\mu.$$

This last inequality is impossible if $\beta<1$, by definition of the pressure. This yields that $\nu_{\eps}$ converges to $\mu_{\K}$, and $h_{\nu_{\eps}}$ converges to $0$. Then \eqref{equ2-equil-potinfi} shows that $\CP(\gamma)\le 0$. 

On the other hand $\CP(\gamma)\ge 0$ because the pressure is larger than the free energy for $\mu_{\K}$, which is zero.  Therefore $\mu_{\K}$ is an equilibrium state.
\end{proof}

In order to use Proposition~\ref{prop-equil-presspos} we need to check that $V_{u}$ satisfies the hypotheses. 

\begin{lemma}
\label{lem-vu-loc-Bowen}
For every cylinder $J$ which does not intersect $\K$, the potential $V_{u}$ satisfies the local Bowen property \eqref{eq:Bowen}.
\end{lemma}
\begin{proof}
Actually, $V_{u}$ satisfies a stronger property: if $x=x_{0}x_{1}\ldots$ and $y=y_{0}y_{1}\ldots$ are in $J$ (a fixed cylinder with $J\cap \K=\emptyset$), if $n$ is their first return time in $J$, and if $x_{k}=y_{k}$ for any 
$0 \le k < n$, then $(S_{n}V_{u})(x)=(S_{n}V_{u})(y)$. 

Assume that $J$ is a $k$-cylinder, and assume without loss of generality that $n > k$. 
The coordinates $x_{j}$ and $y_{j}$ coincide for $0 \le j < n$, but since $J$ is a $k$-cylinder, 
we actually have 
$$x_{j}=y_{j}\text{ for }0\le j\le n+k-1.$$
We recall that $V_{u}$ is constant on sets of the form $(\s\circ H)^{m}(\S)\setminus (\s\circ H)^{m+1}(\K)$. 
Therefore, to compute $V_{u}(z)$ for $z\in\S$ we have to know which
set $(\s\circ H)^{m}(\S)\setminus (\s\circ H)^{m+1}(\K)$ it belongs to. 
Lemma~\ref{lem-sigmacirch} shows that $z=z_{0},z_{1},\ldots$ belongs to $(\s\circ H)^{m}(\S)\setminus (\s\circ H)^{m+1}(\K)$ if and only if $z_{1}\ldots z_{2^{m}}$  coincides with $[\rho_{0}]_{2^{m}}$ or $[\rho_{1}]_{2^{m}}$ and $m$ is the largest integer with this property. 

Let us now study $V_{u}(\s^{j}(x))$ (and $V_{u}(\s^{j}(y))$) for $j$ between 0 and $n-1$. We have to find the largest integer $m$ such that $z_{j+1}\ldots z_{j+1+2^{m}}$ coincides with $[\rho_{0}]_{2^{m}}$ or $[\rho_{1}]_{2^{m}}$.  
As $J$ does not intersect $\K$, the word $x_{n},\ldots ,x_{n+k-1}$ (which is also the word $y_{n},\ldots ,y_{n+k-1}$) is not admissible for $\K$. Therefore, the largest $m$ such that $z_{j+1}\ldots z_{j+1+2^{m}}$ coincides with $[\rho_{0}]_{2^{m}}$ or $[\rho_{1}]_{2^{m}}$ satisfies 
$$2^{m}\le n-j+k-1.$$
In other words, the integer $m$ only depends on the digits where
$\s^{j}(x)$ and $\s^{j}(y)$ coincide. Therefore $V_{u}(\s^{j} (x)) = V_{u}(\s^{j} (y))$. 
\end{proof}

\begin{remark}
An important consequence of Proposition~\ref{prop-exit-mesdeq-vu} 
and Lemma~\ref{lem-vu-loc-Bowen} is that 
the conclusions of Proposition~\ref{prop-equil-presspos} hold.
Although the potential $V_u$ is not continuous
(and in fact undefined at $\s^{-1}(\{ \rho_0, \rho_1\})$), 
it satisfies the local Bowen condition, so that the discontinuity 
is ``invisible'' for the first return map to $J$.
 Proposition~\ref{prop-exit-mesdeq-vu} then implies the existence of an 
equilibrium state.
Furthermore, the critical $z_c(\gamma) \le \CP(\gamma)$. By
a similar argument as used in \cite[Proposition 3.10]{jacoisa}
it can be checked that the conclusion of Lemma~\ref{lem-liminf} holds despite the discontinuity of $V_u$.
$\hfill \blacksquare$\end{remark}

\begin{lemma}
Take $\alpha < 0$ and consider the potential $V_u$ and some cylinder set $J$ disjoint from $\K$.
The critical parameter for the convergence of $(\CL_{z,\gamma}\BBone_{J})(x)$
satisfies
$z_c(\gamma) \ge  2^{-e^{-\gamma \alpha + 2}+1} > 0$ for every 
$\gamma \in \R$ and $x \in J$.
\end{lemma}

\begin{proof}
We now explore the thermodynamic formalism of the unbounded fixed point $V_u$ 
of $\CR$ given by Equation~\ref{eq_exampleVu}.
This potential is piecewise constant, and the value on cylinder sets 
intersecting $\K$ can be pictured 
schematically (with $\alpha = -1$) as follows:
$$
\begin{array}{r|ccccccccccccccccc}
\rho_0 & 1 & 0 & 0 & 1 & 0 & 1 & 1 & 0 & 0 & 1 & 1 & 0 & 1 & 0 & 0 & 1 &\cdots \\
\rho_1 & 0 & 1 & 1 & 0 & 1 & 0 & 0 & 1 & 1 & 0 & 0 & 1 & 0 & 1 & 1 & 0  &\cdots \\
\hline
V_u & 1 & 0 & 1 & -1 & 1 & 0 & 1 & -2 & 1 & 0 & 1 & -1 & 1 & 0 & 1 & -3  &\cdots 
\end{array} 
$$
Here, the third line indicates the value of $V_u$ at $\s^n(\rho_j)$ for $n = 0,1,2,3,\dots$ and $j = 0,1$.
A single ergodic sum of length $b = 2^{k+1}-2^{k-i}$
(with $\alpha < 0$ arbitrary again)
for points $x$ in the same cylinder as $\rho_0$ or $\rho_1$ is
$$
(S_bV_u)(x) = \sum_{j=0}^{2^{k+1}-2^{k-i}-1} V_u(\s^j(x)) = -\alpha(1+i).
$$
Therefore, the contribution of a single excursion
 is 
$$
\Phi_{z,\gamma}(\xi) = \sum_{j=1}^N \sum_{k=0}^{b_j-1}
\gamma \alpha (1-i_j) - b_jz
$$
where $i = i_j$ is such that $b_j = 2^{k+1}-2^{k-i}$.
The contribution to $(\CL_{z,\gamma}\BBone_{J})(x)$ 
of one cluster of excursions then becomes
$E_{z,\gamma}(\xi) \ge \sum_{N \ge 1} A^N$, where (assuming that $z \ge 0$)
\begin{eqnarray*}
A = \sum_{\stackrel{\text{\tiny allowed}}{b \ge 1}}  
e^{\gamma \alpha(1+i) - bz} 
&=& \sum_{k \ge 1}  \sum_{i=0}^{k-1} 
e^{\gamma \alpha(1+i) - (2^{k+1}-2^{k-i})z}\\
&\ge& e^{\gamma \alpha} \sum_{k \ge 1}  \sum_{i=0}^{k-1} 
e^{i \gamma \alpha - 2^{k+1}z}\\
&=& e^{\gamma \alpha}
\sum_{k \ge 1} \frac{1-e^{\gamma \alpha k}}{1-e^{\gamma \alpha}}  e^{- 2^{k+1}z}
\ge e^{\gamma \alpha} \sum_{k \ge 1} e^{- 2^{k+1}z}
\end{eqnarray*}
Take an integer $M \ge e^{-\gamma \alpha +2}$ and $z = 2^{-(M+1)}$.
Then taking only the $M$ first terms of the above sum, we get the the entire sum
is larger than 
$$
e^{\gamma \alpha}  M e^{-2^{M+1} z} \ge e^{\gamma \alpha} 
e^{-\gamma \alpha + 2} e^{-1} = e > 1.
$$
Therefore, 
$A > 1$ and $\sum A^N$ diverges. Hence, the critical $z_c(\gamma) \ge 2^{-e^{-\gamma \alpha + 2}+1} > 0$ for all 
$\gamma > 0$.
\end{proof}

\begin{proof}[Proof of Theorem~\ref{theo-thermo-Vu}]
It is just a consequence of Proposition~\ref{prop-equil-presspos} that a phase transition can only occur at the zero pressure. This never happens, hence the pressure is analytic on $[0, \infty)$ and there is a unique equilibrium state for $-\gamma V_{u}$. 
\end{proof}

\section*{Appendix: The Thue-Morse subshift and the Feigenbaum map}\label{appendix}


The logistic Feigenbaum map $f_\qfeig:I \to I$ is conjugate
to unimodal interval map $f_\cfeig$,
which solves a renormalization equation 
\begin{equation}\label{eq_renorm_feig}
f_\cfeig^2 \circ \Psi(x) = \Psi \circ f_\cfeig(x),
\end{equation}
for all $x \in I$, where $\Psi$ is an affine contraction depending on 
$f_\cfeig$.
Note that $f_\cfeig$ is not a quadratic map, but it has a quadratic critical point $c$. See \cite{Epstein} and \cite[Chapter VI]{MSbook} for an extensive survey.

As a result of \eqref{eq_renorm_feig}, 
$f_\cfeig$ is infinitely renormalizable of Feigenbaum type, \ie 
there is a nested sequence $M_k$ of 
periodic cycles of $2^k$-periodic intervals such that each component
of $M_k$ contains two components of $M_{k+1}$.
The intersection $\CA := \cap_{k \ge 0} M_k$ is a Cantor attractor
on which $f_\cfeig$ acts as a dyadic adding machine.
The renormalization scaling $\Psi:M_k \to M_{k+1}^\crit$, 
where $M_k^\crit$ is the component of $M_k$ containing the critical point, and 
on each $M_k^\crit$ we have
$f_\cfeig^{2^{k+1}} \circ \Psi = \Psi \circ f_\cfeig^{2^k}$.

Furthermore, $\CA$ coincides with the critical $\omega$-limit set $\omega(c)$
and it attracts every point in $I$ except for countably many (pre-)periodic
points of (eventual) period $2^k$ for some $k \ge 0$.
Hence $f_\cfeig:I \to I$ has zero entropy, and the only
probability measures it preserves are
Dirac measures on periodic orbits and a unique measure on $\CA$.
This means that $f_\cfeig:I \to I$ is not very interesting
from a thermodynamic point of view.
However, we can extend $f_\cfeig$ to a quadratic-like map on the 
complex domain, with a chaotic Julia set $\CJ$ supporting topological 
entropy $\log 2$, and its dynamics is a finite-to-one quotient of the full 
two-shift $(\S, \sigma)$.
Equation \eqref{eq_renorm_feig} still holds for the complexification
 $f_\cfeig:U_0 \to V_0$ (a quadratic-like map, to be precise), 
where $\Psi$ is a linear holomorphic 
contraction, and $U_0 \Subset V_0$ are open domains in $\C$ such that
$U_0$ contains the unit interval.
Renormalization in the complex domain thus means that $M_1^\crit$ 
extends to a disks
$U_1 \Subset V_1$ and $f_\cfeig^2:U_1 \to V_1$ is a two-fold branched cover
with branch-point $c$. The {\em little Julia set} 
$$
\CJ_1 = \{ z \in U_1 : f_\cfeig^{2n}(z) \in U_1 \text{ for all } n \ge 0\}
$$ 
is a homeomorphic
copy under $\Psi$ of the entire Julia set $\CJ$, but it should be noted
that most points in $U_1$ eventually leave $U_1$ under iteration of $f_\cfeig^2$:
$U_1$ is not a periodic disk, only the real trace $M_1^\crit = U_1 \cap \R$
is $2$-periodic.
The same structure is found at all scales:
$M_k^\crit = U_k \cap \R$, $U_k \Subset V_k$ and $f_\cfeig^{2^k}:U_k \to V_k$
is a two-fold covering map with little Julia set
$$\CJ_k := 
\{ z \in U_1 : f_\cfeig^{2^kn}(z) \in U_1 \text{ for all } n \ge 0\}
= \Psi(\CJ_{k-1}).
$$
To explain the connection between $f_\cfeig:\CJ \to \CJ$ and symbolic dynamics,
we first observe that the {\em kneading sequence} $\rho$
(\ie the itinerary of the critical value $f_\cfeig(c)$)
is the fixed point of a substitution
$$
H_\feig:\left\{ \begin{array}{l} 0 \to 11, \\ 1 \to 10. \end{array} \right.
$$
Let $\Sigma_\feig = \overline{\orb_\s(\rho)}$ be the corresponding shift space.
If we quotient over the equivalence relation
$x \sim y$ if $x = y$ or $x = w0\rho$ and $y = w1\rho$ (or vice versa) for any
finite and possibly empty word $w$, then 
$\Sigma_\feig \sim$ is homeomorphic to $\CA$, and the itinerary
map $i:\CA \to \Sigma_\feig/\sim$ conjugates $f_\cfeig$ to the shift $\s$.

To make the connection with the Thue-Morse shift, observe that 
the sliding block code $\pi: \S \to \S$ defined by
$$
\pi(x)_k = \left\{ \begin{array}{ll}
1 & \text{ if } x_k \neq x_{k+1}, \\
0 & \text{ if } x_k = x_{k+1},
\end{array} \right.
$$
is a continuous shift-commuting two-to-one covering map.
The fact that it is two-to-one is easily seen because if $x_k = 1-y_k$
for all $k$,
then $\pi(x) = \pi(y)$. Surjectivity can also easily be proved;
once the first digit of $\pi^{-1}(z)$ is chosen, the following digits
are all uniquely determined.
It also transforms the Thue-Morse substitution $H$ into 
$H_\feig$ in the sense that $H_\feig \circ \pi = \pi \circ H$.
For the two Thue-Morse fixed points of $H$ we obtain 
$$
\pi(\rho_0) = \pi(\rho_0) = \rho = 1011101010111011 1011101010111010\dots
$$
Figure~\ref{fig:Thue_Morse_Feigenbaum} summarizes all this in a single 
commutative diagram.

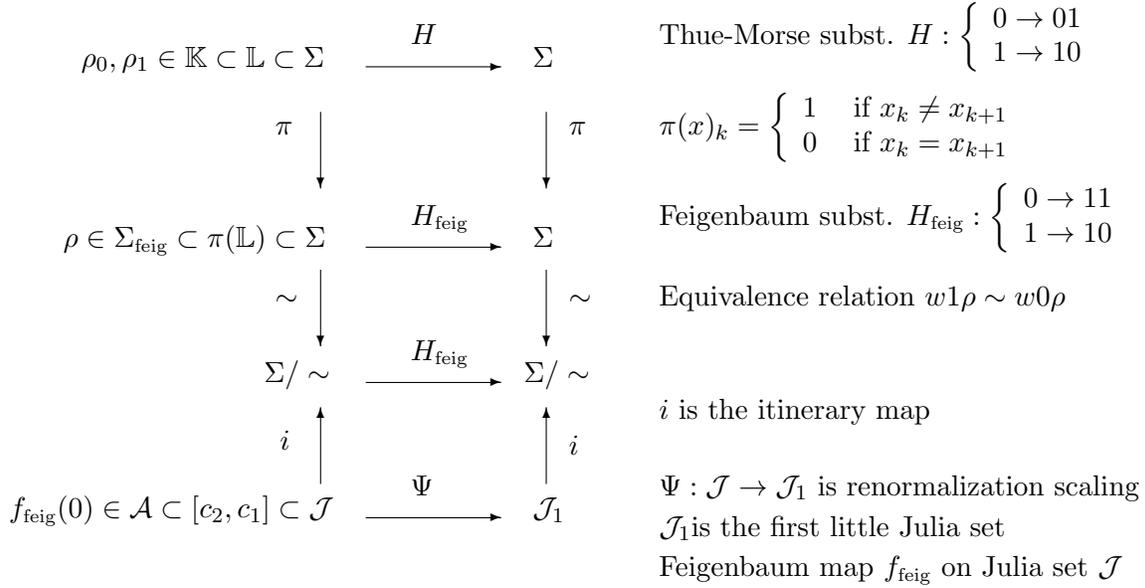
\begin{figure}[ht]
\unitlength=6mm
\begin{picture}(18,14)(-1,0)
\put(-2.8,12){$\rho_0, \rho_1 \in \K \subset \L \subset \S$}
\put(3.5, 12){\vector(1,0){3}}
\put(7.2,12){$\S$}\put(4.5,12.5){$H$}
\put(10,12.5){Thue-Morse subst. 
$H: \left\{ \begin{array}{l} 0 \to 01 \\ 1 \to 10 \end{array} \right.$}
\put(2.5, 11){\vector(0,-1){1.7}}\put(7.5, 11){\vector(0,-1){1.7}}
\put(1.5,10.5){$\pi$}\put(8,10.5){$\pi$}
\put(10,10.5){$\pi(x)_k = \left\{ \begin{array}{ll} 1 & \text{ if } x_k \neq x_{k+1} \\  0 & \text{ if } x_k = x_{k+1} \end{array} \right.$}
\put(-3.2,8){$\rho \in \Sigma_\feig \subset \pi(\L) \subset \S$}
\put(3.5, 8){\vector(1,0){3}}
\put(7.2,8){$\S$}\put(4.5,8.5){$H_\feig$}
\put(10,8.5){Feigenbaum subst.
$H_\feig: \left\{ \begin{array}{l} 0 \to 11 \\ 1 \to 10 \end{array} \right.$}
\put(2.5, 7.5){\vector(0,-1){1.7}}\put(7.5, 7.5){\vector(0,-1){1.7}}
\put(1.5,6.7){$\sim$}\put(8,6.7){$\sim$}
\put(10,6.7){Equivalence relation $w1\rho \sim w0\rho$}
\put(1.25,5){$\S/\sim$} \put(3.5, 5){\vector(1,0){3}}
\put(7,5){$\S/\sim$}\put(4.5,5.5){$H_\feig$}
\put(2.5, 2.75){\vector(0,1){1.7}}\put(7.5, 2.75){\vector(0,1){1.7}}
\put(1.6,3.5){$i$}\put(8,3.4){$i$}
\put(10,4.2){$i$ is the itinerary map}
\put(-4.4,2){$f_\cfeig(0) \in \CA \subset [c_2, c_1] \subset \CJ$} \put(3.5, 2){\vector(1,0){3}}
\put(7.2,2){$\CJ_1$}\put(4.5,2.5){$\Psi$}
\put(10,2.5){$\Psi: \CJ \to \CJ_1$ is renormalization scaling}
\put(10,1.6){$\CJ_1$is  the first little Julia set}
\put(10,0.7){Feigenbaum map $f_\cfeig$ on Julia set $\CJ$}
\end{picture}
\caption{Commutative diagram linking the Thue-Morse substitution shift to the
Feigenbaum map. Further commutative relations:\newline
$\pi$ is continuous, two-to-one and $\s \circ \pi = \pi \circ \s$.\newline
$i:[c_2,c_1] \to \pi(\L)/\sim$ is a homeomorphism and $\s \circ i = i \circ f_\cfeig$.\newline
$\s^2 \circ H = H \circ \s$, \  
$\s^2 \circ H_\feig = H_\feig \circ \s$ and $f_\cfeig^2 \circ \psi = \psi \circ f_\cfeig$.}
\label{fig:Thue_Morse_Feigenbaum}
\end{figure}

The Cantor set $\K$ factorizes over $\Sigma_\feig$ and hence over 
the Cantor attractor $\CA$. The intermediate space $\L$ factorizes over 
the real {\em core} $[c_2, c_1]$ in the Julia set $\CJ$ and we can 
characterize its symbolic dynamics by means of a particular order relation.
Namely, itineraries $i(z)$ of $z \in [c_2, c_1]$ are exactly those sequences that satisfy
$$
\s(\rho) \le_{pl} \s^n \circ i(z) \le_{pl} \rho
\quad \text{ for all } n \ge 0.
$$
Here $\le_{pl}$ is the {\em parity-lexicographical order} by which
$z <_{pl} z'$ if and only if there is a prefix $w$ such that 
$$
\left\{ \begin{array}{l}
z = w0\dots,\quad  z' = w1\dots \text{ and } w \text{ contains an even number of $1$s,}\\[2mm]
z = w1\dots,\quad  z' = w0\dots \text{ and } w \text{ contains an odd number of $1$s.}\end{array} \right.
$$
On the level of itineraries, the substitution $H_\feig$ plays the role
of the conjugacy $\Psi$:
$$
i \circ \Psi(x) = H_\feig \circ i(x) \qquad \text{ for all } x \in [c_2, c_1].
$$
Also let $\le_l$ denote the usual lexicographical order.
\begin{lemma}
Let $[0]$ and $[1]$ denote the one-cylinders of $\S$.
The map $\pi : ([0], \le_l) \to (\S, \le_{pl})$
is order preserving and $\pi : ([1], \le_l) \to (\S, \le_{pl})$ is order reversing.
\end{lemma}
\begin{proof}
First we consider $[0]$ and let $w = 0^n$, then $w0\dots <_l w1\dots$ and
\begin{equation}\label{eq:0w}
\pi(w0\dots) = 0^n\dots \le_{pl}  0^{n-1}1\dots = \pi(w1\dots).
\end{equation}
Now if we change the $k$-th digit in $w$ (for $k \ge 2$), 
then still $w0 <_l w1$
and both the $k$-th and $k-1$-st digit of $\pi(w\dots)$ change. 
This does not affect the parity of $1$s in $\pi(w)$ and so \eqref{eq:0w}
remains valid. Repeating this argument, we obtain that $\pi$ is 
order-preserving for all words $w$ starting with $0$.

Now for the cylinder $[1]$ and $w = 10^{n-1}$, we find 
$w0\dots <_l w1\dots$ and
\begin{equation*}\label{eq:1w}
\pi(w0\dots) = 10^{n-1}\dots \ge_{pl}  10^{n-2}1\dots = \pi(w1\dots).
\end{equation*}
The same argument shows that $\pi$ reverses order for all words $w$ 
starting with $1$.
\end{proof}

This lemma shows that $\pi^{-1} \circ i([c_2, c_1])$ consists of 
the sequence $s$ such that for all $n$,
$$
\left\{ \begin{array}{rcccll}
\s(\rho_1)  &\le_{l} & \s^n(s) & \le_{l} & \rho_1 & \text{ if } 
\s^n(s) \text{ starts with } 1, \\[2mm]
\rho_0 & \le_{l} & \s^n(s) & \le_{l} & \s(\rho_0) & \text{ if } 
\s^n(s) \text{ starts with } 0.
\end{array} \right.
$$
However, the class of sequence carries no shift-invariant measures
of positive entropy, and the thermodynamic formalism reduces to finding
measures that maximize the potential. The measure supported furthest away
from $\K$ is the Dirac measure on $\overline{01}$ 
(with $\pi(\overline{01}) = \overline{1}$).

Instead, if we look at the entire Julia set $\CJ$,
the combination of $\pi$ and the quotient map do not decrease
entropy, and the potential
$-\log |f'_\cfeig|$ has thermodynamic interest for 
the complexified Feigenbaum map $f_\cfeig:\CJ \to \CJ$.
Since $\Psi$ is affine, differentiating \eqref{eq_renorm_feig}
and taking logarithms, we find that
$$
\CR_\feig(\log |f'_\feig|) := \log |f'_\feig|  \circ f_\cfeig \circ \psi + \log |f'_\cfeig| \circ \Psi = \log |f'_\cfeig|,
$$
so $V_\feig := \log |f'_\cfeig|$ is a fixed point of the renormalization
operator $\CR_\feig$ mimicking $\CR$.
Furthermore, since $U_k = \Psi^{k-1}(U_1)$,
its size is exponentially small in $k$ and hence there is some fixed
$\alpha < 0$ such that
 $V_\cfeig \approx \alpha(k-1)$ on $U_k \setminus U_{k+1}$.
Since $U_k \setminus U_{k+1}$ corresponds
to the cylinder $(\s \circ H)^{k-1} \setminus (\s \circ H)^k$, the potential
$V_u$ from Section~\ref{subsec-unbounded1} is comparable to $V_\cfeig$.
As shown in Section~\ref{subsec-unbounded2}, $V_u$ exhibits no phase 
transition.

The following proposition for complex analytic maps is stated in general 
terms, but proves the phase transition of Feigenbaum maps in particular.

\begin{proposition}\label{prop:feig_PT}
Let $f:\C \to \C$ be an $n$-covering map without parabolic periodic points
such that the omega-limit set $\omega(\Crit)$ of the critical set
is nowhere dense in its Julia set $\CJ$, and such that
there is some $c \in \Crit$ such that $f:\omega(c) \to \omega(c)$ 
has zero entropy and Lyapunov exponent.
Then for $\phi = \log|f'|$ and every $\gamma > 2$, $\CP(-\gamma \phi) = 0$.
\end{proposition}

\begin{proof}
As $f$ has no parabolic points, $\lambda_0 := 
\inf\{ |(f^n)'(p)| \ : \ p \in \CJ \text{ is an $n$-periodic point} \} > 1$.
Obviously, all the invariant measures $\mu$ supported on $\omega(c)$ 
have $h_{\mu} -\gamma \int \log |f'| d\mu = 0$, so $\CP(-\gamma \phi) \ge 0$.

To prove the other inequality, we fix $\gamma > 2$ and for some $f$-invariant measure $\mu$, we choose a neighborhood $U$ intersecting $\CJ$
but bounded away from $\orb(\Crit)$ such that $\mu(U) > 0$.
We can choose $\diam(U)$ so small compared to the distance $d(\orb(\Crit), U)$
that $K^{\gamma-1} < \lambda_0^{\gamma-2}$ where $K$ is the distortion constant
in the Koebe Lemma, see \cite[Theorem 1.3]{Pomm}. Since $K \to 1$ as 
$\kappa := \diam(U)/d(\orb(\Crit), U) \to 0$, 
we can satisfy the condition on $K$ by choosing $U$ small enough.

Let $F:\cup_i U_i \to U$ be the first return map to $U$. 
Each branch $F|_{U_i} =  f^{\tau_i}|_{U_i}$, with first return time 
$\tau_i > 0$
 can be extended holomorphically to $f^{\tau_i}:V_i \to f^{\tau_i}(U_i)$
where $f^{\tau_i}(V_i)$ contains a disc around $f^{\tau_i}(U_i)$
of diameter $\ge d(\orb(\Crit), U) \ge \diam(f^{\tau_i}(U_i))/\kappa$.
Hence the Koebe Lemma implies that the distortion of $f^{\tau_i}|_{U_i}$ is 
bounded by $K = K(\kappa)$.
Furthermore, since each $U_i$ contains a $\tau_i$-periodic point
of multiplier $\ge \lambda_0$, we have $\diam(U_i)/\diam(U) \le K/\lambda_0$. Therefore, for any $x \in U$,
\begin{eqnarray*}
\CL_{0,\gamma}(\BBone_{J})(x)
&=& \sum_{i, \exists x' \in U_i\ F(x') = x} |F'(x')|^{-\gamma} \\
&\le& \sum_i K \left(\frac{\diam(U_i)}{\diam(U)}\right)^{\gamma} \\
&\le& \sum_i
K \ \frac{\area(U_i)}{\area(U)} 
\left( \frac{K}{\lambda_0} \right)^{\gamma-2} \\
&\le& K^{\gamma-1} \lambda_0^{-(\gamma-2)} \sum_i \frac{\area(U_i)}{\area(U)}. 
\end{eqnarray*}
Since the regions $U_i$ are pairwise disjoint, the sum in the final 
line $\le 1$, so our choice of $K$
gives that the above quantity is bounded by $1$.
Therefore the radius of convergence $\lambda_{0,\gamma} \le 1$.
Taking the logarithm and using
Abramov's formula, we find that the pressure $\CP(-\gamma \phi) \le 0$.
\end{proof}

Department of Mathematics\\
University of Surrey\\
Guildford, Surrey, GU2 7XH\\
United Kingdom\\
\texttt{H.Bruin@surrey.ac.uk}\\
\texttt{http://personal.maths.surrey.ac.uk/st/H.Bruin/}
\\[3mm]
D\'epartement de  Math\'ematiques\\
Universit\'e de Brest\\
6, avenue Victor Le Gorgeu\\
C.S. 93837, France \\
\texttt{Renaud.Leplaideur@univ-brest.fr}\\
\texttt{http://www.math.univ-brest.fr/perso/renaud.leplaideur}

\end{document}